\documentclass[11pt,reqno]{amsart}
\usepackage[T1]{fontenc}
\usepackage{amsmath,amsthm,amssymb,mathrsfs}
\usepackage{lmodern,microtype,mathtools,bm,enumitem,needspace}
\usepackage{xcolor}
\usepackage{hyperref}
\usepackage{bookmark}
\usepackage[nameinlink,capitalize,noabbrev]{cleveref}
\microtypesetup{expansion=false}
\setlist[itemize]{leftmargin=2.2em,itemsep=0.15em,topsep=0.25em}
\newtheorem{theorem}{Theorem}[section]
\newtheorem{lemma}[theorem]{Lemma}
\newtheorem{proposition}[theorem]{Proposition}
\theoremstyle{definition}
\newtheorem*{AIuse}{AI Usage Statement}
\theoremstyle{remark}
\newcommand{\Hn}{\mathbb H^n}
\newcommand{\Haus}{\mathcal H}
\newcommand{\LTheta}{{L_\Theta}}
\newcommand{\gTheta}{g_1}
\DeclareMathOperator{\spt}{spt}
\newcommand{\RR}{\mathbb R}
\newcommand{\eps}{\varepsilon}
\newcommand{\dd}{\,\mathrm d}
\newcommand{\grad}{\nabla_H}
\newcommand{\divh}{\operatorname{div}_H}
\newcommand{\Shape}{A^H}
\newcommand{\En}{\mathcal E}
\newcommand{\BC}{\mathcal B^\eps}
\newcommand{\Ca}{c}
\newcommand{\AngleError}{\mathsf E_u}
\newcommand{\ShapeError}{\mathsf E_A}
\newcommand{\PSphere}{\mathcal S}

\newcommand{\CentreError}{\mathsf E_C}
\newcommand{\Aone}{A^1}
\newcommand{\Aeps}{A^\eps}
\newcommand{\Amer}{A^1_{00}}
\newcommand{\Amodel}{A^P}
\newcommand{\Hhor}{H_0}
\newcommand{\Hone}{H^1}
\newcommand{\Heps}{H^\eps}
\newcommand{\Hpr}{h}
\newcommand{\None}{N^1}
\newcommand{\Neps}{N^\eps}
\newcommand{\Dil}{\mathcal D}
\newcommand{\Xvp}{X_F^{\mathrm{vp}}}
\newcommand{\Qhor}{Q_\Sigma}
\newcommand{\Qonehor}{Q^H}
\newcommand{\Qoneeps}{Q^\eps}
\newcommand{\vT}{v^1_T}
\newcommand{\vF}{v^1_{X_F}}
\newcommand{\vK}{v^1_{X_K}}
\DeclareMathOperator{\Id}{Id}
\newcommand{\Ric}{\operatorname{Ric}}
\newcommand{\ip}[2]{\langle #1,#2\rangle}
\newcommand{\ResultBookmark}[3]{%
  \ifcsname r@#2\endcsname
    \bookmark[level=#1,dest={\getrefbykeydefault{#2}{anchor}{}}]{#3 \ref*{#2}: \nameref*{#2}}%
  \fi
}
\numberwithin{equation}{section}
\allowdisplaybreaks[1]
\title[Quantitative Centre Rigidity]{Quantitative Centre Rigidity under\\
Riemannian Approximation in Heisenberg Groups}
\author{Wei-Chun Chen}
\address{Department of Mathematics, National Taiwan University, Taipei 10617, Taiwan}
\email{b14201028@ntu.edu.tw}

\author{Wei-Ting Kao}
\address{Institute of Mathematics, Academia Sinica, Taipei 106319, Taiwan}
\email{wtkao@as.edu.tw}

\author{Chun-Kai Lien}
\address{Department of Mathematics, National Taiwan University, Taipei 10617, Taiwan}
\email{d13221002@ntu.edu.tw}
\hypersetup{hypertexnames=false,hidelinks,
  pdftitle={Quantitative Centre Rigidity under Riemannian Approximation in Heisenberg Groups},
  pdfauthor={Wei-Chun Chen, Wei-Ting Kao, and Chun-Kai Lien},
  pdfsubject={Quantitative centre rigidity for stationary stable prescribed-curvature boundaries},
  pdfkeywords={Heisenberg group, Pansu sphere, Riemannian approximation,
    prescribed curvature, volume-constrained stability, horizontal perimeter,
    characteristic set, centre reconstruction}}
\date{}
\subjclass[2020]{Primary 53C17, 53C24; Secondary 49Q10, 53C42}
\keywords{Heisenberg group, quantitative rigidity, Riemannian approximation,
  volume-constrained stability}
\begin{document}
\begin{abstract}
For $n\ge2$, we prove a quantitative centre-rigidity estimate for
bounded sets of finite Riemannian perimeter that are stationary and
volume-constrained stable for normalized Riemannian perimeter with
prescribed curvature $h_j$.
Under confinement to a fixed compact set and cutoffs of vanishing full surface
Dirichlet energy on the smooth regular boundary, the squared $L^2$
norms of the horizontal centre derivative, angle defect and weighted
shape defect are bounded by
$C(\eps_j^2+\|h_j-2n\|_{C^2})|\Omega_j|$ when the parenthesized
quantity is sufficiently small. The excess of normalized Riemannian
perimeter over horizontal perimeter satisfies the same bound.
\end{abstract}
\maketitle

\section{Introduction}\label{sec:introduction}

On the $(2n+1)$-dimensional Heisenberg group $\Hn$, let
$g_\eps$ make the standard frame $X_k,Y_k$ ($1\le k\le n$)
and $\eps T$ orthonormal. Horizontal lengths stay fixed,
while the length of the vertical field $T$ is $\eps^{-1}$.
For a smooth bounded region $\Omega$, the normalized area and volume are
\[
 P_\eps(\Omega)=\eps\operatorname{Area}_{g_\eps}(\partial\Omega),
 \qquad |\Omega|=\eps\operatorname{Vol}_{g_\eps}(\Omega).
\]
For each fixed set $\Omega$ of finite Riemannian perimeter,
$P_\eps(\Omega)\downarrow P_0(\Omega)=P_H(\Omega)$ as $\eps\downarrow0$.
At fixed $\eps>0$, volume-constrained stationarity and stability for
$P_\eps$ agree with those of $g_\eps$-area.

In $\mathbb H^1$, Ni~\cite{Ni} studied invariant Riemannian CMC
surfaces and their sub-Riemannian limits. Torralbo and
Urbano~\cite{TorralboUrbano} proved volume-preserving stability of
Riemannian CMC spheres. Franceschi, Montefalcone and
Monti~\cite{FranceschiMontefalconeMonti} constructed explicit CMC sphere
families converging to Pansu spheres and obtained quantitative area
comparisons for volume-preserving perturbations supported in vertical
half-cylinders. The estimate below concerns stationary stable
boundaries in $\Hn$, $n\ge2$, and controls their centre, angle and
shape defects.

For a prescribed curvature function $\Hpr$, stationarity of
\[
 \Omega\longmapsto P_\eps(\Omega)-\int_\Omega \Hpr\dd x
\]
gives $\Heps=\Hpr$ on a smooth boundary, where $\Heps$ is the trace
mean curvature in $g_\eps$. The Hessian of this functional is restricted
to normal speeds with zero volume derivative. The prescribed function may vary
in space; its deviation from $2n$, the horizontal mean curvature of
the unit Pansu sphere, enters the estimate below.

On a unit Pansu sphere, the centre reconstructed from each point
and its outward normal is constant~\cite[Section~2]{Rigidity}.
Write the reconstruction as $\Ca_j=(z_{c,j},t_{c,j})$ and set
$\mathsf E_{C,j}=D_H\Ca_j$ and
$\mathsf E_{u,j}=\nabla_Hu_j+e_{n,j}$, where $u_j$ is the normal angle
and $e_{n,j}$ is the characteristic direction.
The defect $\mathsf E_{A,j}$ is the difference between the horizontal
shape operator and the unit-sphere operator.
These quantities are defined in Section~\ref{sec:setup};
$D_H$ denotes horizontal tangential differentiation and
$d\sigma_H$ is horizontal perimeter measure.
Put
\[
 \Delta P_\eps(\Omega)=P_\eps(\Omega)-P_0(\Omega)\ge0,
 \qquad \Delta P_j=\Delta P_{\eps_j}(\Omega_j),\qquad
 \delta_j=\eps_j^2+\|\Hpr_j-2n\|_{C^2(U)}.
\]
On $\{\cos u_j\ge1/4\}$, the dilation-correction coefficients are
$O(\eps_j^2)$; the weighted cap integral in
\eqref{eq:scaling-excess-weight} is bounded by $\tfrac83\Delta P_j$.
For stationary, volume-constrained stable families satisfying
\textup{(H1)--(H2)} of Section~\ref{sec:stable-families},
Theorem~\ref{thm:absorption} gives constants $C,\delta_*>0$ such that
\[
 \delta_j\le\delta_*\quad\Longrightarrow\quad
 \int_{\Sigma_j}\bigl(|\mathsf E_{C,j}|^2
       +|\cos u_j\,\mathsf E_{A,j}|^2+|\mathsf E_{u,j}|^2\bigr)
       \dd\sigma_{H,j}+\Delta P_j\le C\delta_j|\Omega_j|.
\]
The constants depend only on the dimension and the fixed ambient sets.
The shape defect carries the factor $\cos u_j=|N_{H,j}|$,
which vanishes at characteristic points.

The nonnegative energy
$\En=\int(|D_Hz_c|^2+|\mathsf E_u|^2)\dd\sigma_H$ controls all three
defects by the centre derivative bound and the square decomposition
in Section~\ref{sec:centre-energy}.
The centre identity of \cite[Theorem~4.1]{Rigidity} and horizontal
curvature linearization of \cite[Lemma~2.8(i)]{JacobiSpectrum}
relate this energy to a Jacobi density.
A three-profile contraction compares that density with the elliptic
$P_\eps$ Hessian and retains a negative meridional curvature square.
Correcting the test field's volume flux makes constrained stability
applicable. The prescribed-curvature residuals are bounded by
$C\|\Hpr_j-2n\|_{C^2(U)}(|\Omega_j|+\En_j)$.
Combining the cap contribution with the negative dilation excess
and absorbing the energy term gives the estimate.

The local angle and Jacobi identities hold on the noncharacteristic boundary.
Proposition~\ref{sel:prop:physical-form} gives finite curvature
energy and cutoff closure for admissible speeds.
Combining the centre, curvature and central-variation densities
under a common cutoff yields \eqref{eq:stationary-contraction}.

\begin{AIuse}
We used ChatGPT 6.0 Astra to assist in extending the principal
computations of our companion paper on stable rigidity to the
setting of Riemannian approximation and in carrying out complex
calculations. The main ideas of this paper were developed by the
authors. All AI-assisted mathematical derivations and calculations
were independently verified by the authors. The authors reviewed
and approved the final manuscript and take full responsibility
for its mathematical content.
\end{AIuse}

\section{Geometry, elliptic perimeter, and admissible variations}
\label{sec:setup}\label{sec:stable-families}\label{sec:centre-reconstruction}

The point and outward normal of a regular hypersurface determine a
candidate centre of a unit Pansu sphere. Horizontal derivatives of
that centre and the normal angle define $\En$; the horizontal shape
identity compares $\En$ with the weighted shape defect. The elliptic
perimeter, stability and capacity assumptions make these quantities and the
normal speeds used in second variation integrable and admissible.

\subsection{Normal angle and centre reconstruction}
\label{subsec:normal-angle-centre}

Let $T=\partial_t$ and identify horizontal vectors with their
coefficients in the frame
\[
 X_k=\partial_{x_k}+y_k\partial_t,\qquad
 Y_k=\partial_{y_k}-x_k\partial_t,\qquad [X_k,Y_k]=-2T,\qquad JX_k=Y_k.
\]
The group law is
$(z,t)\cdot(w,\tau)=(z+w,t+\tau-\langle Jz,w\rangle)$.
The contact form is
\[
 \Theta=dt-\sum_k(y_k\dd x_k-x_k\dd y_k).
\]
Put $\xi=\ker\Theta$. Thus $d\Theta|_\xi=2\sum_kdx_k\wedge dy_k$, and
Haar measure is Lebesgue measure $\dd x$ in these coordinates.
We write $\mathcal H$ for Hausdorff measures; see \cite{FSSC}
for horizontal perimeter.

The Levi metric on $\xi$ is $\LTheta=\tfrac12d\Theta(\cdot,J\cdot)$.
For $\eps>0$ put $g_\eps=\LTheta+\eps^{-2}\Theta^2$, so
$g_1=\LTheta+\Theta^2$ and $dV_{\gTheta}=\dd x$.
On a smooth hypersurface $\Sigma$, put
$d\sigma_1=d\Haus_{\gTheta}^{2n}|_\Sigma$.
Write $\Neps$ for the outward $g_\eps$-unit normal,
$\nabla^\eps$ for its Levi--Civita connection,
$\Aeps(X)=\nabla^\eps_X\Neps$, $\Heps=\operatorname{tr}\Aeps$, and
$d\sigma_\eps$ for its surface measure. Surface gradients are
$\nabla^\eps_\Sigma$; inner products and norms default to $g_1$
except on $g_\eps$ tensors.

For the outward $g_1$-unit normal, set
$N^1_T=\langle\None,T\rangle$ and $N_H=\None-N^1_TT$. Put
$\Sigma^*=\{x\in\Sigma:|N_H(x)|>0\}$. The horizontal perimeter measure
and the normal data on $\Sigma^*$ are
\begin{equation}
d\sigma_H=|N_H|\,d\sigma_1,\qquad
 e_{2n}=-N_H/|N_H|,\qquad e_n=-Je_{2n},
 \qquad \alpha=N^1_T/|N_H|.
\label{sel:tag-2-1}
\end{equation}
Write $\grad$ for the horizontal tangential gradient on $\Sigma^*$.
Integrals against $d\sigma_H$ of quantities defined only on $\Sigma^*$
are taken there: the characteristic set has zero $d\sigma_H$ measure.
Set
\begin{equation}\label{eq:angle}
 \begin{aligned}
 &u=\arctan\alpha,\qquad \cos u=|N_H|>0,\qquad \sin u=N^1_T,\\
 &\None=-(\cos u)e_{2n}+(\sin u)T,\qquad
 e_0=(\sin u)e_{2n}+(\cos u)T.
 \end{aligned}
\end{equation}

For an ambient field $X$, define its horizontal and Riemannian normal
speeds and its volume flux by
\[
 v_X^H=\langle X,-e_{2n}+\alpha T\rangle,\qquad
 v_X^\eps=\langle X,\Neps\rangle_{g_\eps},\qquad
 I_X=\delta V[X].
\]
The speed $v_X^1=\langle X,\None\rangle$ satisfies
\[
 v_X^1=\cos u\,v_X^H,\qquad
 v_X^H\,d\sigma_H=v_X^1\,d\sigma_1
                 =\eps v_X^\eps\,d\sigma_\eps,\qquad
 I_X=\int v_X^H\dd\sigma_H.
\]
The speed-measure product extends across characteristic points as the
volume-flux measure.

Write $\PSphere_a$ for the unit-radius Pansu sphere centred at $a$.
The candidate centre determined by $x=(z,t)$ and its normal data is
\cite[Section~2]{Rigidity}
\begin{align}
 z_c&=z+{\cos^{2} u}e_{2n}-{\sin u}{\cos u} e_n,\notag\\
 t_c&=t+\langle Jz_c,z\rangle
       -\frac{u+{\sin u}{\cos u}}{2},\qquad
 \Ca=(z_c,t_c).
 \label{eq:centres}
\end{align}
For $x=(z,t)$, a horizontal unit vector $v$, and $|u|\le\pi/2$,
define $a=(z_a,t_a)$ by
\eqref{eq:centres} with $e_{2n}=-v$ and $e_n=Jv$.
Since $v,Jv$ are orthonormal, $x\in\PSphere_a$:
\begin{equation}\label{eq:profile}
 |z-z_a|={\cos u},\qquad
 t+\langle Jz_a,z\rangle-t_a
                      =\frac{u+{\sin u}{\cos u}}{2}.
\end{equation}
For ${\cos u}>0$, the radial direction is
$e=(z-z_a)/|z-z_a|={\cos u}v+{\sin u}Jv$, and the sphere's outward horizontal
normal is $\cos u\,e-\sin u\,Je=v$.
At $\cos u=0$, the formula gives its two poles.

The formula $z_c=z-|N_H|N_H-N^1_TJN_H$ is $C^1$ at $N_H=0$.
For $t_c$, use $u=\arcsin N^1_T$ and
\[
 \frac{d}{ds}\bigl(\arcsin s+s\sqrt{1-s^2}\bigr)
 =2\sqrt{1-s^2}\longrightarrow0\qquad(s\to\pm1).
\]
Extending this scalar function constantly past $\pm1$ makes it $C^1$,
so $\Ca$ extends in $C^1$ across the characteristic set
\cite[Lemma~2.3]{Rigidity}.

On the unit Pansu sphere,
$\grad u=-e_n$; set
\begin{equation}\label{eq:angle-defect}
 \AngleError=\grad u+e_n.
\end{equation}
For $z_c$ and $\Ca$, $D_H$ is componentwise differentiation along
$\xi\cap T\Sigma^*$, with Hilbert--Schmidt norm from $\gTheta$ and
the Euclidean target metrics; put $\CentreError=D_H\Ca$.
The energy density and its integral are
\begin{equation}\label{eq:defect}
 \mathfrak e:=|D_Hz_c|^2+|\AngleError|^2,\qquad \En=\int\mathfrak e\dd\sigma_H.
\end{equation}

On a compact set $\mathcal K\subset\Hn$, the vertical centre identity
\cite[(2.11)]{Rigidity} and Cauchy--Schwarz give, for an orthonormal
frame $(E_a)$ of $\xi\cap T\Sigma^*$,
\begin{equation}\label{eq:centre-derivative-bound}
 \begin{aligned}
 |\CentreError|^2
 &=|D_Hz_c|^2+\sum_a|\langle J(E_az_c),z\rangle
                  -\cos^2u\langle\AngleError,E_a\rangle|^2\\
 &\le C(\mathcal K)\mathfrak e\qquad((z,t)\in\mathcal K).
 \end{aligned}
\end{equation}

\subsection{Shape and angle defects}\label{sec:centre-energy}

On the horizontal bundle $\xi$, write the flat Tanaka--Webster
connection as $\nabla^\xi$, determined by
$\nabla^\xi_U X_k=\nabla^\xi_U Y_k=0$ for every $U$
\cite[Section~2]{JacobiSpectrum}. On an oriented smooth hypersurface
$\Sigma$, let $\nabla^{1,\Sigma}$ be the induced Levi--Civita surface
connection.

For $X\in\xi\cap T\Sigma^*$, the symmetric horizontal shape operator
\cite[Proposition~2.2]{CCHY} and its trace are
\[
 \Shape(X)=-\nabla^\xi_Xe_{2n}+\alpha(JX-\ip{X}{e_n}e_{2n}),
 \qquad \Hhor=\operatorname{tr}\Shape.
\]
The Riemannian shape $\Aone(X)=\nabla^1_X\None$ has the block form
\begin{equation}\label{eq:blocks}
 \Aone=\begin{pmatrix}\Aone_{HH}&\AngleError\\\AngleError^\top&\Amer\end{pmatrix}
 \quad\hbox{on }(\xi\cap T\Sigma^*)\oplus\RR e_0,
\end{equation}
where $\Amer=e_0u$ is the curvature in the $e_0$ (meridional) direction and
\begin{equation}\label{eq:geometry}
 \Aone_{HH}={\cos u}\Shape,\qquad
 \Hone=\operatorname{tr}\Aone={\cos u}\Hhor+\Amer,\qquad
 \operatorname{div}^1_\Sigma e_0=-{\sin u}\Hhor.
\end{equation}
Differentiating the normal also gives
$\nabla^1_\Sigma u=\Aone e_0-e_n=\AngleError-e_n+\Amer e_0$.

The unit Pansu-sphere model is
$\Amodel X=X+\langle X,e_n\rangle e_n$, so
$\operatorname{tr}\Amodel=2n$ and $|\Amodel|^2=2n+2$.
Set $\ShapeError=\Shape-\Amodel$.
The square identity \cite[(2.9)]{Rigidity} gives, on $\Sigma^*$,
\[
 \begin{aligned}
 |D_Hz_c|^2
 &=\sum_a|{\cos u}\ShapeError E_a
                 +\langle\AngleError,E_a\rangle({\cos u}e_n+{\sin u}e_{2n})|^2\\
 &=|{\cos u}\ShapeError|^2+|\AngleError|^2
             +2{\cos^{2} u}\langle\AngleError,\ShapeError e_n\rangle.
 \end{aligned}
\]
Here $(E_a)$ is a horizontal tangent orthonormal frame and
$\ShapeError$ is self-adjoint. With
$(\AngleError\otimes e_n)X=\langle e_n,X\rangle\AngleError$,
adding $|\AngleError|^2$ and completing the square yields
\begin{equation}\label{eq:defect-square}
 \begin{gathered}
 \mathfrak e=|{\cos u}\ShapeError+{\cos u}(\AngleError\otimes e_n)|^2
                  +(2-{\cos^{2} u})|\AngleError|^2,\\
 |{\cos u}\ShapeError|^2+|\AngleError|^2\le3\mathfrak e.
 \end{gathered}
\end{equation}
Indeed, $|\AngleError\otimes e_n|=|\AngleError|$ and
$\cos^2u\le2-\cos^2u$ give
\[
 |\AngleError|^2\le\mathfrak e,\qquad
 |\cos u\,\ShapeError|^2
 \le2|\cos u(\ShapeError+\AngleError\otimes e_n)|^2
       +2\cos^2u|\AngleError|^2\le2\mathfrak e.
\]
Thus $\mathfrak e=0$ forces $\AngleError=\ShapeError=0$ on $\Sigma^*$.
Since ${\cos u}\ShapeError=\Aone_{HH}-{\cos u}\Amodel$ and $\grad u=\AngleError-e_n$,
\eqref{eq:blocks} and \eqref{eq:defect-square} give
\begin{equation}\label{eq:basic-curvature-domination}
 \mathfrak e+|{\cos u}\ShapeError|^2+|\AngleError|^2+|\grad u|^2
 \le C(1+|\Aone|^2).
\end{equation}

\subsection{Elliptic perimeter and admissible families}
\label{subsec:elliptic-families}

Finite Riemannian perimeter means that the distributional gradient
of $\mathbf 1_\Omega$ in $g_1$ has finite total variation.
Write $\partial^*\Omega$ for the reduced boundary and $\None$ for its
measure-theoretic outward $g_1$-unit normal. The perimeter measure is
$\Haus_{g_1}^{2n}|_{\partial^*\Omega}$, and Gauss--Green gives
\[
 \int_\Omega\operatorname{div}_{\gTheta}X\dd x
 =\int_{\partial^*\Omega}\langle X,\None\rangle
       \dd\Haus_{\gTheta}^{2n}
\]
for every compactly supported smooth ambient field $X$.
Using $\None=N_H+N^1_TT$, define
\begin{equation}\label{eq:elliptic-perimeter}
 P_\eps(\Omega)=\int_{\partial^*\Omega}
 \sqrt{|N_H|^2+\eps^2(N^1_T)^2}\dd\Haus_{\gTheta}^{2n},
 \qquad P_0(\Omega)=P_H(\Omega).
\end{equation}
Thus $P_0(\Omega)=\int_{\partial^*\Omega}|N_H|\dd\Haus_{\gTheta}^{2n}$.
The reduced boundary and normal remain fixed as $\eps$ varies.
At $N_H=0$ the density is $\eps$, whereas the horizontal density
vanishes. The smooth boundary carrying perimeter in \textup{(H2)}
is an additional hypothesis.
The homogeneous dimension is $2n+2$.

Fix a compact set $\mathcal K$ inside a relatively compact open
set $U\subset\Hn$.
The $C^k(U)$ norms use the fixed ambient coordinates.
\begin{enumerate}[label=\textup{(H\arabic*)},leftmargin=3em]
\item \textbf{Confinement and prescribed-curvature bounds.}\label{hyp:bounds}
Each $\Omega_j\subset\mathcal K$ has finite Riemannian perimeter and
\[
 |\Omega_j|>0,\qquad P_0(\Omega_j)>0,\qquad 0<\eps_j\le1.
\]
The prescribed mean-curvature function $\Hpr_j\in C^2(U)$ satisfies
$\|\Hpr_j-2n\|_{C^2(U)}\le1$.

\Needspace{13\baselineskip}
\item \textbf{Smooth boundary and cutoffs.}\label{hyp:capacity}
An outward-oriented smooth hypersurface $\Sigma_j\subset\partial^*\Omega_j$
carries all the Riemannian perimeter of $\Omega_j$:
\[
 \Haus_{\gTheta}^{2n}(\partial^*\Omega_j\setminus\Sigma_j)=0.
\]
Characteristic points may belong to $\Sigma_j$.
Put $\sigma_{1,j}=\Haus_{\gTheta}^{2n}|_{\Sigma_j}$ and
$\Sigma_j^*=\{x\in\Sigma_j:|N_H(x)|>0\}$.
For each fixed $j$, there are
$\eta_{j,\ell}\in C_c^1(\Sigma_j^*)$ with $0\le\eta_{j,\ell}\le1$
and $\eta_{j,\ell}\to1$ $\sigma_{1,j}$-almost everywhere on $\Sigma_j$,
such that
\begin{equation}\label{eq:capacity-assumption}
 \int_{\Sigma_j}|\nabla^1_{\Sigma_j}\eta_{j,\ell}|^2\dd\sigma_{1,j}
 \longrightarrow0\qquad(\ell\to\infty).
\end{equation}
\end{enumerate}

By \textup{(H2)}, $\sigma_{1,j}(\Sigma_j\setminus\Sigma_j^*)=0$
(zero-extend the cutoffs).
Stationarity of $P_{\eps_j}(\Omega)-\int_\Omega\Hpr_j\dd x$
under compactly supported smooth ambient variations in $U$ gives,
on $\Sigma_j$,
\begin{equation}\label{eq:euler-assumption}
 H^{\eps_j}_j=\Hpr_j.
\end{equation}

The cofactor formulas give
\begin{equation}\label{sel:eq:elliptic-metric}
 dV_{g_{\eps_j}}=\eps_j^{-1}dx,\qquad
 P_{\eps_j}(\Omega)=\eps_j\operatorname{Area}_{g_{\eps_j}}(\partial^*\Omega).
\end{equation}
Thus $P_\eps$ is normalized Riemannian area. At fixed $\eps_j>0$,
the volume constraint is unchanged by the constant factor in
$dV_{g_{\eps_j}}$.

Constants $C$ depend only on $n,\mathcal K,U$;
$C_j$ may depend on $j$.
For fixed $j$, write $\Omega=\Omega_j$, $\Sigma=\Sigma_j$,
$V=|\Omega|$, $\eps=\eps_j$ and $\Delta P_\eps=\Delta P_\eps(\Omega)=\Delta P_j$.
The prescribed function remains $\Hpr_j$ on $U$, while $\Heps$
denotes geometric mean curvature on $\Sigma_j$.
All $C^k$ norms of $\Hpr_j-2n$ are taken on $U$.

\subsection{Admissible variations}
\label{sec:admissible-variations}

Let $\mathcal B^{\eps_j}_{1,j}$ be the Hessian of the functional
$\Omega\mapsto P_{\eps_j}(\Omega)-\int_\Omega \Hpr_j\dd x$
on smooth compactly supported
$\gTheta$-normal speeds in $\Sigma_j^*$, extended to compactly supported
$C^1$ speeds there by its local Jacobi form. Define the energy norm
\begin{equation}\label{eq:fixed-physical-form-norm}
 \|v\|_j^2=\int_{\Sigma_j^*}\bigl\{|\nabla^1_{\Sigma_j}v|^2
                  +(1+|\Aone_j|^2)v^2\bigr\}\dd\sigma_{1,j}.
\end{equation}
Call a bounded $C^1$ speed $v$ on $\Sigma_j^*$ \emph{admissible} if
$|\nabla^1_{\Sigma_j}v|\le C_j(1+|\Aone_j|)$ for some $C_j$.

\begin{proposition}[Curvature energy and admissible normal speeds]
\label{sel:prop:physical-form}
\ResultBookmark{2}{sel:prop:physical-form}{Proposition}%
Fix $j$ in a family satisfying \textup{(H1)--(H2)}. Assume the
functional above is stationary at $\Omega_j$ under compactly
supported smooth ambient variations in $U$, and
\[
 \mathcal B^{\eps_j}_{1,j}(v,v)\ge0
 \quad\text{for every }v\in C_c^2(\Sigma_j^*)
 \text{ with }\int_{\Sigma_j}v\dd\sigma_{1,j}=0.
\]
Then the Riemannian second fundamental form $\Aone_j$ satisfies
\begin{equation}\label{sel:tag-3-4}
 \int_{\Sigma_j}(1+|\Aone_j|^2)\dd\sigma_{1,j}<\infty.
\end{equation}
For admissible $v,w$, the cutoffs $\eta_{j,\ell}$ in \textup{(H2)} satisfy
$\|\eta_{j,\ell}v-v\|_j\to0$ and similarly for $w$.
The resulting finite symmetric pairing is independent of the
choice of common cutoffs:
\begin{equation}\label{sel:tag-3-6}
 \mathcal B^{\eps_j}_{1,j}(v,w)=\lim_{\ell\to\infty}
 \mathcal B^{\eps_j}_{1,j}(\eta_{j,\ell}v,\eta_{j,\ell}w).
\end{equation}
If $\int v\dd\sigma_{1,j}=0$, then $\mathcal B^{\eps_j}_{1,j}(v,v)\ge0$.
The same cutoffs satisfy
\begin{equation}\label{eq:cap-est}
\int_{\Sigma_j}(1+|\Aone_j|)|\nabla^1_{\Sigma_j}\eta_{j,\ell}|\dd\sigma_{1,j}\longrightarrow0.
\end{equation}
For smooth compactly supported ambient fields $X,Y$ in $U$, the
ambient Hessian equals
$\mathcal B^{\eps_j}_{1,j}(\langle X,\None\rangle,
\langle Y,\None\rangle)$. This identification extends by continuity
to bounded $C^1$ surface traces with
$|\nabla^1_{\Sigma_j}X|+|\nabla^1_{\Sigma_j}Y|
\le C_j(1+|\Aone_j|)$.
\end{proposition}

The curvature-energy and closure assertions follow from the
fixed-$\eps_j$ Jacobi inequality in
Appendix~\ref{app:admissible-variations}. The same appendix identifies
this closure with the ambient Hessian on the surface traces of the
fields used below; the scalar norm retains its curvature weight.

For each admissible $v$, the volume-speed measures
$\eta_{j,\ell}v\dd\sigma_{1,j}$ converge to $v\dd\sigma_{1,j}$ in total variation.
Quadratic forms use the same cutoff on every speed;
linear identities use $\eta_{j,\ell}^2$. The order of limits is
\[
 \ell\to\infty\quad(j\text{ fixed}),\qquad j\to\infty.
\]
Equation~\eqref{eq:cap-est} also applies to $\eta_{j,\ell}^2$, since
$|\nabla^1_\Sigma(\eta_{j,\ell}^2)|\le2|\nabla^1_\Sigma\eta_{j,\ell}|$.
Equation~\eqref{sel:tag-3-4} makes any density bounded by
$C_j(1+|\Aone_j|^2)d\sigma_{1,j}$ integrable. Equation~\eqref{eq:cap-est}
removes fluxes bounded by
$C_j(1+|\Aone_j|)|\nabla^1_{\Sigma_j}\eta_{j,\ell}|d\sigma_{1,j}$.

Equations~\eqref{eq:basic-curvature-domination} and \eqref{sel:tag-3-4}
give $\En<\infty$.
Using $\Aone_{HH}=\cos u(\Amodel+\ShapeError)$ and
$\grad u=\AngleError-e_n$, the square identity
\eqref{eq:defect-square} gives
\begin{align}
 \int(|{\cos u}\ShapeError|^2+|\AngleError|^2)\dd\sigma_H
 &\le3\En,\label{eq:coercive}\\
 \int(|\Aone_{HH}|^2+|\AngleError|^2+|\grad u|^2)\dd\sigma_H
 &\le C(P_0(\Omega)+\En).\label{eq:bound}
\end{align}
The cutoff closure in Proposition~\ref{sel:prop:physical-form} makes
these bounds available for the global second-variation pairings.

\section{The volume-constrained test field}
\label{app:finite-master}

On the unit Pansu sphere, the horizontal dilation speed is
$F(\alpha)=1+\alpha\arctan\alpha=1+u\alpha$.
As in \cite[Section~3]{Rigidity}, the same profile $F$ defines the
ambient field $X_F$ on a general boundary; it need not equal the
dilation field.
Its flux $I_F$ need not vanish, so $X_F-(I_F/I_{\Dil})\Dil$ is the
field to which constrained stability applies. The auxiliary profile
$K$ and central translation $T$ form a Hessian contraction whose
horizontal density is related to the centre energy by
\eqref{eq:local-noncmc-balance}. The Riemannian correction retains
a negative meridional curvature square.

Let $\Dil=z\cdot\partial_z+2t\partial_t$ generate
$s\mapsto\Dil_{e^s}$, where $\Dil_r(z,t)=(rz,r^2t)$ for $r>0$.
The fields $\Dil,T$ used in global variations
are cut off in $U$ and agree with the original fields near
$\mathcal K$. Hence $\operatorname{div}\Dil=2n+2$ near $\mathcal K$, and since
$\Omega\subset\mathcal K$,
\[
 I_{\Dil}=\delta V[\Dil]=(2n+2)V>0.
\]
The horizontal gradient $\grad$ and divergence $\divh$ correspond to
$\nabla_b$ and $\operatorname{div}_H$ in \cite{Rigidity}.

\subsection{The dilation speed on the unit Pansu sphere}
\label{subsec:canonical-profiles}

Set
\[
 q_H =|\Shape|^2+4e_n \alpha+2(n+1)\alpha^2,\qquad
 \divh Z=\operatorname{div}_{H,\Sigma}Z+2\alpha\ip{Z}{e_n }.
\]
Here $Z$ is horizontal tangent, $\operatorname{div}_{H,\Sigma}$ is the
horizontal connection trace, and $\divh$ is divergence for $d\sigma_H$.
The characteristic-direction identity in \cite[Section~4.1]{Rigidity}
gives $\divh e_n=2n\alpha$.
\Needspace{4\baselineskip}
For $f,g\in C_c^1(\Sigma^*)$, define the horizontal scalar form
\[
 \Qhor(f,g)=\int
       \{\langle\grad f,\grad g\rangle-q_H fg\}\dd\sigma_H.
\]
At constant $\Hhor$ this is the Jacobi pairing in
\cite[Section~3]{Rigidity}. For variable $\Hhor$, the curvature
linearization has the additional transport term
in \eqref{eq:normal-horizontal-comparison}.

Differentiating $F$ gives
\[
 \partial_\alpha F(\alpha)=u+{\sin u}{\cos u},\qquad
 \partial_\alpha^2F(\alpha)=2{\cos^{4} u}.
\]
$\partial_\alpha F$ is twice the model height in
\eqref{eq:profile}, as in \cite[Section~3.2]{Rigidity}.

\subsection{The test field and volume correction}

Fix a smooth cutoff $\chi:[0,1]\to[0,1]$ equal to one on $[0,1/2]$
and zero on $[2/3,1]$. As in \cite[Lemma~3.4]{Rigidity}, set
$X_F=\chi(\cos u)(F/\alpha)T-(1-\chi(\cos u))Fe_{2n}$ on $\Sigma^*$.
Quotients by $\alpha$ are used only where $\chi\ne0$, so $\alpha\ne0$.
On each cap $0<\cos u\le2/3$, write $s=N^1_T=\sin u$, so $u=\arcsin s$. Then
\[
 \frac F\alpha=\arcsin s+\frac{\sqrt{1-s^2}}s,\qquad
 \frac{d}{ds}\left(\frac F\alpha\right)=-\frac{\sqrt{1-s^2}}{s^2}.
\]
As $s\to\pm1$, the value tends to $\pm\pi/2$ and the derivative to zero.
Constant extension past each endpoint is therefore $C^1$.
Composing with $N^1_T$, the cutoff construction gives a local $C^1$
ambient field through characteristic points, with horizontal speed $F$.

Set \cite[Section~4]{Rigidity}
\[
 \mathsf E_{u,n}=\langle\AngleError,e_n\rangle,\qquad
 \beta=\int \mathsf E_{u,n}\dd\sigma_H.
\]

Write $\BC(X,Y):=\mathcal B^{\eps_j}_{1,j}(\langle X,\None\rangle,\langle Y,\None\rangle)$
for the bilinear Hessian of
$\Omega\mapsto P_\eps(\Omega)-\int_\Omega \Hpr_j\dd x$ on fields
$X,Y$, before imposing the volume constraint.
Set
\[
 I_F=\delta V[X_F]=\int F\dd\sigma_H,\qquad
 \kappa=I_F/I_{\Dil},\qquad \Xvp=X_F-\kappa \Dil.
\]

\begin{lemma}[Volume-corrected field and volume flux]
\label{lem:global-domain}
\ResultBookmark{3}{lem:global-domain}{Lemma}%
Under the hypotheses of Proposition~\ref{sel:prop:physical-form},
the normal speeds of
$T,X_F,\Dil$ are admissible. The field $\Xvp$ has zero volume flux and
constrained stability applies to its speed. Moreover,
\begin{equation}\label{eq:global-support}
 \begin{gathered}
 2n I_F=(2n+1)P_0(\Omega)-\beta,\qquad
 I_F^2\le CP_0(\Omega)(P_0(\Omega)+\En),\\
 I_F+\kappa^2P_0(\Omega)+\kappa^2V+|\kappa|I_F
 \le C(P_0(\Omega)+\En)\le C(V+\En).
 \end{gathered}
\end{equation}
\end{lemma}
\begin{proof}
The identity $\nabla^1_\Sigma u=\Aone e_0-e_n$ following
\eqref{eq:blocks} gives
$|\nabla^1_\Sigma u|\le1+|\Aone|$. From
$\None=-(\cos u)e_{2n}+(\sin u)T$, the cutoff terms in $X_F$ give
\[
 \vT=\sin u,\qquad
 \vF=\chi(\cos u)(F/\alpha)\sin u
       +(1-\chi(\cos u))F\cos u
     =\cos u+u\sin u.
\]
The quotient occurs only where $\alpha\ne0$. Differentiating these
speeds on $\Sigma^*$ yields
\[
 \nabla^1_\Sigma\vT=(\cos u)\nabla^1_\Sigma u,\qquad
 \nabla^1_\Sigma\vF=(u\cos u)\nabla^1_\Sigma u.
\]
Thus $|\vT|\le1$, $|\vF|\le1+\pi/2$, and both gradients are at most
$(\pi/2)(1+|\Aone|)$. For $v^1_{\Dil}=\langle\Dil,\None\rangle$ and
every $g_1$-unit tangent vector $E$,
\[
 |v^1_{\Dil}|\le|\Dil|,\qquad
 E(v^1_{\Dil})=\langle\nabla^1_E\Dil,\None\rangle
                   +\langle\Dil,\Aone E\rangle.
\]
Since $\Dil$ and $\nabla^1\Dil$ are bounded on $\mathcal K$,
\[
 |\nabla^1_\Sigma v^1_{\Dil}|
 \le\|\nabla^1\Dil\|_{L^\infty(\mathcal K)}
    +\|\Dil\|_{L^\infty(\mathcal K)}|\Aone|
 \le C_j(1+|\Aone|).
\]
All three speeds are bounded and $C^1$ on $\Sigma^*$, hence admissible.
Using $d\sigma_H=(\cos u)d\sigma_1$ and \textup{(H1)--(H2)},
\[
 \int|F|\dd\sigma_H=\int|\vF|\dd\sigma_1
 \le(1+\pi/2)\sigma_1(\Sigma)<\infty.
\]
As $I_{\Dil}=(2n+2)V>0$, $\kappa$ is finite, so
$\langle\Xvp,\None\rangle=\vF-\kappa v^1_{\Dil}$ is admissible.
Its volume flux is $I_F-\kappa I_{\Dil}=0$; thus
\cref{sel:prop:physical-form} gives constrained stability.

For $I_F$, the product rule gives
\[
 \divh(ue_n)=e_nu+u\divh e_n
 =(\mathsf E_{u,n}-1)+2nu\alpha=\mathsf E_{u,n}+2n F-(2n+1).
\]
Testing with $\eta_\ell^2$ gives
\[
 \int\eta_\ell^2(\mathsf E_{u,n}+2n F-(2n+1))\dd\sigma_H
 =-\int u\langle e_n,\grad(\eta_\ell^2)\rangle\dd\sigma_H.
\]
Because $|u|\le\pi/2$, $|e_n|=1$,
$d\sigma_H=(\cos u)d\sigma_1\le d\sigma_1$, and
$|\grad\psi|\le|\nabla^1_\Sigma\psi|$, the flux is at most
$\pi\int|\nabla^1_\Sigma\eta_\ell|\dd\sigma_1\to0$ by
$0\le\eta_\ell\le1$ and \eqref{eq:cap-est}.
The remaining density $\mathsf E_{u,n}\,d\sigma_H$
is integrable by \eqref{eq:basic-curvature-domination} and
\eqref{sel:tag-3-4}, proving the first identity in \eqref{eq:global-support}.

Hence
$|2n I_F|\le(2n+1)P_0(\Omega)+\sqrt{P_0(\Omega)\En}$ by Cauchy--Schwarz
and \eqref{eq:defect-square}, giving the bound for $I_F^2$.
Set $\Lambda_\eps=\eps\partial_\eps P_\eps$ with the set fixed.
Its density is $\eps^2\sin^2u/\Phi_\eps$ relative to $d\sigma_1$,
so $0\le\Lambda_\eps\le P_\eps$. Stationarity under dilation gives
\begin{equation}\label{eq:stationary-size-bound}
 2nP_\eps\le(2n+1)P_\eps-\Lambda_\eps
 =\int_\Omega(\Dil\Hpr_j+(2n+2)\Hpr_j)\dd x\le CV.
\end{equation}
Here the scaling law
$P_\eps(\Dil_{e^s}\Omega)=e^{(2n+1)s}P_{\eps e^{-s}}(\Omega)$
and the compactly supported extension of $\Dil$ justify the first
variation directly for finite-perimeter sets. Thus
$P_0(\Omega)/V\le C$. Since $\kappa=I_F/((2n+2)V)$,
this ratio bound and $\sqrt{P_0(\Omega)\En}\le(P_0(\Omega)+\En)/2$
give the remaining estimates in \eqref{eq:global-support}.

Finally $F(\alpha )\ge(\pi/2)|\alpha |$: for $\alpha\ne0$ their difference is
$1-|\alpha |\arctan(1/|\alpha |)\ge0$ by $\arctan t\le t$, so
$\int|\alpha |\dd\sigma_H\le2I_F/\pi$.
\end{proof}

\subsection{The Riemannian Jacobi contraction}

The cancellation
$F^2-\tfrac{\pi^2}{4}\alpha^2-\alpha K=\cos^2u$
forces the contracted squared $g_1$-normal speed to be $\cos^4u$;
differentiating the resulting profiles gives the angular coefficient
$-3\cos^4u$ in \eqref{eq:coeff}. Thus set
\[
 K(\alpha)=\frac{F^2-\tfrac{\pi^2}{4}\alpha^2-{\cos^{2} u}}{\alpha},\quad K(0)=0.
\]
The numerator is smooth, even, and zero at $\alpha=0$; hence it is
$\alpha^2$ times a smooth function and $K$ extends smoothly.

Use $X_K=\chi(\cos u)(K/\alpha)T-(1-\chi(\cos u))Ke_{2n}$
on $\Sigma^*$, with the same cutoff $\chi$. On the caps, with $s=\sin u$,
\[
 \begin{aligned}
 \frac K\alpha
 &=(F/\alpha)^2-\tfrac{\pi^2}{4}-\frac{\cos^4u}{\sin^2u}
   =u^2-\tfrac{\pi^2}{4}+2u\frac{\cos u}{\sin u}+\cos^2u,\\
 \frac{d}{ds}\left(\frac K\alpha\right)
 &=\frac{2\cos u(s\cos u-u)}{s^2}\longrightarrow0\qquad(s\to\pm1).
 \end{aligned}
\]
The endpoint values are zero, so $K/\alpha$ also extends in $C^1$ as a
function of $N^1_T$, giving a local $C^1$ ambient field $X_K$.
Since $|s|\ge\sqrt5/3$ on the caps, integration of the derivative gives
\[
 |K/\alpha|\le C\int_{|s|}^1\sqrt{1-t^2}\dd t
       \le C(1-s^2)^{3/2},\qquad |K|\le C\cos^2u.
\]
The horizontal speeds are $\alpha,F,K$. Substituting
$\alpha=\tan u$ into $\cos u\,K$, and extending through $u=0$,
gives the $g_1$-normal speeds
\begin{equation}\label{eq:profiles}
 \vT={\sin u},\qquad \vF={\cos u}+u{\sin u},\qquad
 \vK=2u{\cos u}+(u^2-\tfrac{\pi^2}{4}+{\cos^{2} u}){\sin u}.
\end{equation}
Differentiating with respect to $u$ gives
\[
 \partial_u \vT={\cos u},\qquad \partial_u \vF=u{\cos u},\qquad
 \partial_u \vK={\cos u}(u^2-\tfrac{\pi^2}{4}+3\cos^2u).
\]
Since $|u|<\pi/2$ and $\nabla^1_\Sigma u=\Aone e_0-e_n$,
\[
 |\vK|\le\pi+\tfrac{\pi^2}{4}+1,\qquad
 |\nabla^1_\Sigma\vK|
 \le(\tfrac{\pi^2}{4}+3)(1+|\Aone|).
\]
Thus $\vK$ is a bounded $C^1$ admissible speed, as are
$\vT,\vF$ by Lemma~\ref{lem:global-domain};
\cref{sel:prop:physical-form} applies to all three.

For any symmetric bilinear form $L$ on normal speeds, set
\[
 [L]_*:=L(\vF,\vF)-\tfrac{\pi^2}{4}L(\vT,\vT)-L(\vT,\vK).
\]
For horizontal speeds, replace $(\vF,\vT,\vK)$ by
$(F,\alpha,K)$; for fields, replace them by $(X_F,T,X_K)$.
The notation $[L]_{*,\eta}$ applies $\eta$ to every argument.

For scalar triples $f=(f_T,f_F,f_K)$ and $g=(g_T,g_F,g_K)$,
the product bracket means the symmetric contraction
\[
 [fg]_*:=f_Fg_F-\tfrac{\pi^2}{4}f_Tg_T
             -\tfrac12(f_Tg_K+f_Kg_T).
\]
The defining relation for $K$ and the derivatives of \eqref{eq:profiles} give
\begin{equation}\label{eq:coeff}
 \begin{aligned}
 [v^1v^1]_*&=\cos^2u\bigl(F^2-\tfrac{\pi^2}{4}\alpha^2-\alpha K\bigr)
                   =\cos^4u,\\
 [\partial_uv^1\,\partial_uv^1]_*
   &=\cos^2u\bigl(u^2-\tfrac{\pi^2}{4}
                -(u^2-\tfrac{\pi^2}{4}+3\cos^2u)\bigr)
     =-3\cos^4u,\\
 [v^1\partial_uv^1]_*
   &=\tfrac12\partial_u[v^1v^1]_*
     =-2\sin u\cos^3u.
 \end{aligned}
\end{equation}
For $i\in\{T,X_F,X_K\}$, the speeds satisfy
$v_i^H=v_i^1/\cos u$, and hence
$\partial_uv_i^H=(\partial_uv_i^1)/\cos u+\sin u\,v_i^1/\cos^2u$.
Therefore
\[
 \begin{aligned}
 [v^Hv^H]_*&=(\cos u)^{-2}[v^1v^1]_*=\cos^2u,\\
 [\partial_uv^H\,\partial_uv^H]_*
   &=-3\cos^2u-4\sin^2u+\sin^2u=-3,\\
 [v^H\partial_uv^H]_*
   &=-2\sin u\cos u+\sin u\cos u=-\sin u\cos u.
 \end{aligned}
\]
Contracting the local integrand of $\Qhor$ gives the density
\begin{equation}\label{eq:combined}
 j^H:=-(3|\grad u|^2+q_H {\cos^{2} u}).
\end{equation}

By the centre identity \eqref{eq:local-noncmc-balance}, this density is
$-\mathfrak e-|\grad u|^2-2n$ plus the variable-curvature and
divergence terms. The zero-flux field $\Xvp$ makes the corresponding
elliptic Hessian accessible to constrained stability.
\Needspace{18\baselineskip}
\section{The centre identity and quantitative rigidity}
\label{app:non-cmc}\label{sec:proof-main}

Theorem~\ref{thm:absorption} bounds the sum of the squared
$L^2(d\sigma_{H,j})$ norms of $D_H\Ca_j$,
$\cos u_j\,\mathsf E_{A,j}$ and $\mathsf E_{u,j}$ by $C\delta_j|\Omega_j|$.

\begin{theorem}[Quantitative centre rigidity]\label{thm:absorption}
\ResultBookmark{2}{thm:absorption}{Theorem}%
Let $n\ge2$ and let $(\Omega_j,\eps_j,\Hpr_j)$ satisfy
\textup{(H1)--(H2)} of Section~\ref{sec:stable-families}.
For each $j$, suppose the first variation at $\Omega_j$ of
\[
 \Omega\longmapsto P_{\eps_j}(\Omega)-\int_\Omega \Hpr_j\dd x
\]
vanishes under compactly supported smooth ambient variations in $U$.
Suppose also that its Hessian is nonnegative on $\gTheta$-normal
speeds $v\in C_c^2(\Sigma_j^*)$ satisfying
$\int_{\Sigma_j}v\dd\sigma_{1,j}=0$.
There are constants $C,\delta_*>0$, depending only on
$n,\mathcal K,U$, such that whenever
$\eps_j^2+\|\Hpr_j-2n\|_{C^2(U)}\le \delta_*$,
\begin{equation}\label{eq:full-centre-decay}
 \int_{\Sigma_j}\bigl(
 |\mathsf E_{C,j}|^2+|{\cos u_j}\mathsf E_{A,j}|^2+|\mathsf E_{u,j}|^2
 \bigr)\dd\sigma_{H,j}+\Delta P_j
 \le C\bigl(\eps_j^2+\|\Hpr_j-2n\|_{C^2(U)}\bigr)|\Omega_j|.
\end{equation}
\end{theorem}

For fixed $j$, suppress its index and set
\[
 \begin{aligned}
 \mathcal G_j&:=\int_{\Sigma_j}\bigl(|\mathsf E_{C,j}|^2
       +|\cos u_j\,\mathsf E_{A,j}|^2+|\mathsf E_{u,j}|^2\bigr)
       \dd\sigma_{H,j},\\
 \delta_j&:=\eps^2+\|\Hpr_j-2n\|_{C^2(U)},\\
 \mathcal R_{\Dil}(Y)&:=\BC(\Dil,Y)+2n\delta V[Y],\qquad
 \mathsf D:=2nI_{\Dil}-(2n+1)P_0(\Omega).
 \end{aligned}
\]
Let $\mathcal R$ be given by \eqref{eq:Rexpand}, with combined density
$\mathfrak r_\eps$ from \eqref{eq:joint-curvature-density}.
\[
 \begin{gathered}
 \mathcal R_1:=\int\mathfrak r_\eps\dd\sigma_H,\qquad
 \mathcal R_2:=-2\kappa\mathcal R_{\Dil}(X_F),\qquad
 \mathcal R_3:=\mathcal R-\mathcal R_1-\mathcal R_2,\\
 \mathcal M_\eps:=\int_{\Sigma^*}
       4\eps^2(1+\eps^2\alpha^2)^{-3/2}|\Amer|^2\dd\sigma_H.
 \end{gathered}
\]
\[
 \begin{aligned}
 \mathcal R_1&\le-\tfrac12\mathcal M_\eps
       +C(\eps^2+\|\Hpr_j-2n\|_{C^1(U)})(V+\En),\\
 \mathcal R_2&\le\tfrac38\mathcal M_\eps+\tfrac83\kappa^2\Delta P_\eps
       +C(\eps^2+\|\Hpr_j-2n\|_{C^1(U)})(V+\En),\\
 \mathcal R_3&\le-(2n+1)\kappa^2\Delta P_\eps
       +C\|\Hpr_j-2n\|_{C^2(U)}(V+\En).
 \end{aligned}
\]
Since $2n+1>8/3$, their sum and the master inequality give
\[
 \En+\tfrac18\mathcal M_\eps
       +(2n+1-\tfrac83)\kappa^2\Delta P_\eps
 \le C_0\delta_j(V+\En).
\]
The flux bounds in Lemma~\ref{lem:global-domain} and the positive
lower bound for $\kappa$ proved below yield
$\mathcal G_j+\Delta P_j\le C\delta_j|\Omega_j|$ after absorption.

\Needspace{6\baselineskip}
For $P_0-2nV$, \cite[Theorem~4.1]{Rigidity} identifies the negative
volume-constrained Hessian with
$\En+\int|\AngleError|^2\dd\sigma_H
-\beta^2/((2n+1)P_0(\Omega))$. For $P_\eps$, stationarity gives
$\Heps=\Hpr_j$, while $\Hhor$ need not equal $2n$. The resulting
curvature and divergence terms are combined with the Riemannian
Jacobi density under common cutoffs across the characteristic set.

The integral in \eqref{eq:full-centre-decay} is taken over the
noncharacteristic part of $\Sigma_j$. The factor $\cos u_j$ weights
the shape defect near characteristic points; the estimate gives no
unweighted $L^2$ bound for that defect.

\subsection{The centre identity with variable curvature}

The local centre identity expresses $\mathfrak e$ in terms of
$q_H$, $\Hhor-2n$ and a horizontal divergence.
The variable-curvature and divergence terms enter the global pairing
\eqref{eq:stationary-contraction}. The divergence is interpreted
weakly on $\Sigma^*$, where the angle and horizontal shape operator
are defined.

Since $\Hhor-2n=\operatorname{tr}\ShapeError$, set
$\omega=\partial_\alpha F=u+{\sin u}{\cos u}$ and
\begin{equation}\label{eq:ach}
 a_\omega=2{\cos^{2} u}-e_n\omega-(2n-2)\alpha\omega
       =2{\cos^{2} u}(2-\mathsf E_{u,n})-2(n-1)\alpha\omega.
\end{equation}
\Needspace{9\baselineskip}
\begin{proposition}[Local centre identity]
\label{prop:noncmc-local-centre}
\ResultBookmark{3}{prop:noncmc-local-centre}{Proposition}%
Let $\Sigma\subset\Hn$, $n\ge2$, be an oriented two-sided $C^2$
hypersurface. With the notation above, the following identity holds
weakly on $\Sigma^*$, including at $\alpha=0$:
\begin{equation}\label{eq:local-noncmc-balance}
 \begin{gathered}
 \mathfrak e=2|\grad u|^2+{\cos^{2} u}q_H-2n
       -(\Hhor-2n)a_\omega+\divh Z_\omega,\\
 Z_\omega=\omega\bigl(\ShapeError-(\Hhor-2n)\Id\bigr)e_n.
 \end{gathered}
\end{equation}
\end{proposition}
\begin{proof}
The smooth Codazzi contraction computed in
\cite[Appendix~A]{Rigidity} gives
$\divh(\Shape e_n)=e_n\Hhor+2\alpha \Hhor$.
The identities $\Amodel e_n=2e_n$ and $\divh e_n=2n\alpha$ give
\[
 \begin{aligned}
 \divh(\ShapeError e_n)
 &=e_n\Hhor+2\alpha \Hhor-2(2n\alpha)
   =e_n\Hhor+2\alpha(\Hhor-2n),\\
 \divh\bigl((\Hhor-2n)e_n\bigr)
 &=e_n\Hhor+2n\alpha(\Hhor-2n),\\
 \divh\bigl((\ShapeError-(\Hhor-2n)\Id)e_n\bigr)
 &=-(2n-2)\alpha(\Hhor-2n).
 \end{aligned}
\]
On each compact noncharacteristic parametrized patch, approximate
the $C^2$ surface by smooth $\Sigma_m$ converging in $C^2$.
Set $Z_m=(\mathsf E_{A,m}-(H_{0,m}-2n)\Id)e_{n,m}$.
For a fixed compactly supported $C^1$ test $\psi$,
\[
 \int\langle Z_m,\nabla_{H,m}\psi\rangle\dd\sigma_{H,m}
 =-\int\psi\,\operatorname{div}_{H,m}Z_m\dd\sigma_{H,m}
 =(2n-2)\int\psi\alpha_m(H_{0,m}-2n)\dd\sigma_{H,m}.
\]
After pullback to the fixed compact patch, the densities and both
integrands converge uniformly, since they depend continuously on the
second jet and the patch is noncharacteristic. Passing to the limit
gives the same identity for
$Z=(\ShapeError-(\Hhor-2n)\Id)e_n$.
Since $\omega=\partial_\alpha F(\alpha)$, the chain rule gives
\[
 \grad\omega=2{\cos^{4} u}\grad\alpha=2{\cos^{2} u}\grad u,\qquad
 e_n\omega=2{\cos^{2} u}\langle \grad u,e_n\rangle=2{\cos^{2} u}(\mathsf E_{u,n}-1).
\]
Testing the weak identity for $Z$ with $\omega\psi$ gives
$\divh(\omega Z)=\langle\grad\omega,Z\rangle+\omega\divh Z$.
Expanding this product yields
\[
 \divh Z_\omega
 =2{\cos^{2} u}\langle \grad u,\ShapeError e_n\rangle
  -(\Hhor-2n)\bigl(e_n\omega+(2n-2)\alpha\omega\bigr).
\]
Expanding $q_H$ using
$\langle \Amodel,\ShapeError\rangle=\Hhor-2n+\langle\ShapeError e_n,e_n\rangle$
and $e_n\alpha=(\cos u)^{-2}(\mathsf E_{u,n}-1)$ gives
\[
 {\cos^{2} u}q_H
 =2(n-1)+|{\cos u}\ShapeError|^2
       +2{\cos^{2} u}\langle\ShapeError e_n,e_n\rangle+4\mathsf E_{u,n}
       +2{\cos^{2} u}(\Hhor-2n).
\]
Together with \eqref{eq:defect-square} and $\grad u=\AngleError-e_n$, this gives
\[
 \mathfrak e-(2|\grad u|^2+{\cos^{2} u}q_H-2n)
 =2{\cos^{2} u}\bigl(\langle \grad u,\ShapeError e_n\rangle-(\Hhor-2n)\bigr).
\]
The product rule and \eqref{eq:ach} give
\eqref{eq:local-noncmc-balance}.
\end{proof}

\subsection{Riemannian Jacobi comparison}\label{sec:jacobi}

The centre identity gives a horizontal Jacobi density, while stability
controls the Hessian of $P_\eps-\int_\Omega\Hpr_j\dd x$. Their
three-profile contractions differ by the elliptic density $j_\eps$
below. Its negative meridional square is retained when the residual
$\mathcal R$ is estimated.

For $v,w\in C_c^2(\Sigma^*)$, define
\[
\begin{aligned}
 \Qonehor(v,w)&=D^2P_0^\perp(v,w)
       -\int\Hhor\Hone vw\dd\sigma_1,\\
 \Qoneeps(v,w)&=D^2P_\eps^\perp(v,w)
       -\int\Heps\Hone vw\dd\sigma_1.
\end{aligned}
\]
These Hessians use the $g_1$-normal graph
$x\mapsto\exp_x^{g_1}((sv+rw)\None)$.
At $s=r=0$, the normal speed in $r$ is $w$ with zero $s$-derivative,
and $\partial_s\dd\sigma_1=\Hone v\dd\sigma_1$. Thus
\[
 \Qonehor(v,w)=\int w\dot\Hhor[v\None]\dd\sigma_1,
 \qquad
 \Qoneeps(v,w)=\int w\dot\Heps[v\None]\dd\sigma_1.
\]

The Hessian of $\int_\Omega\Hpr_j\dd x$ is
$\int(\partial_{\None}\Hpr_j+\Hpr_j\Hone)vw\dd\sigma_1$.
Subtracting it and using stationarity, $\Heps=\Hpr_j$, gives
\[
 \BC(v\None,w\None)=\Qoneeps(v,w)
              -\int(\partial_{\None}\Hpr_j)vw\dd\sigma_1.
\]
On each compact subset of $\Sigma^*$, these Hessians have bounded
coefficients and involve at most first derivatives of $v,w$;
they therefore extend continuously to $C_c^1(\Sigma^*)$.
Put
\begin{equation}\label{eq:anisotropy}
 \Phi_\eps=\sqrt{{\cos^{2} u}+\eps^2{\sin^{2} u}},\qquad
 \mu_\eps={\cos u}/\Phi_\eps=(1+\eps^2\alpha^2)^{-1/2}.
\end{equation}
The $P_\eps$ density relative to $d\sigma_1$ is $\Phi_\eps$, and
$0<\mu_\eps\le1$ on $\Sigma^*$.
The identities $T+\alpha e_{2n}=e_0/{\cos u}$,
$(T+\alpha e_{2n})\alpha=\Amer/{\cos^3u}$, and
$\Hone={\cos u}\Hhor+\Amer$ reduce
\cite[Corollary~3.5]{Verzellesi2026} to
\begin{equation}\label{eq:He}
 \begin{aligned}
 \Heps&=\mu_\eps \Hhor+\frac{\eps^2}{\Phi_\eps^3}\Amer,\\
 \Heps-\Hhor
 &=\frac{\mu_\eps-1}{\cos u}(\Hone-\Amer)+\frac{\eps^2}{\Phi_\eps^3}\Amer.
 \end{aligned}
\end{equation}
Direct differentiation gives
$\partial_u\mu_\eps=-\eps^2\sin u/\Phi_\eps^3$.

On a compact noncharacteristic patch, let $j_\eps$ be the
$d\sigma_H$-density obtained by applying $[\cdot]_*$ to the local
integrand of $\Qoneeps-\Qonehor$ on $(\vT,\vF,\vK)$.

\begin{proposition}[Jacobi comparison for the three profiles]
\label{prop:common}
\ResultBookmark{3}{prop:common}{Proposition}%
The density $j_\eps$ has the pointwise formula
\begin{align}
 j_\eps={}&
 {\cos^{2} u}(1-\mu_\eps)
 (|\Aone_{HH}|^2+4|\AngleError |^2-4 \langle\AngleError,e_n\rangle+2(n+1){\sin^{2} u})\notag\\
 &-\eps^2\mu_\eps^3(|\grad u|^2+4|\Amer|^2).\label{eq:star}
\end{align}
\end{proposition}

\begin{proof}
For the Levi--Civita connection of $\gTheta$,
$\nabla^1_\Sigma u=\Aone e_0-e_n$ and
$\Ric^1(\None,\None)=2[n-(n+1){\cos^{2} u}]$. Here $\Ric^1$ and $\operatorname{Sec}^1$
are the ambient curvatures of $\gTheta$; the displayed Ricci formula
uses $\Ric^1|_\xi=-2\gTheta|_\xi$ and $\Ric^1(T,T)=2n$
\cite[Proposition~3.3, $\eps=1$]{Verzellesi2026}.
With $\Delta^1_\Sigma=\operatorname{div}^1_\Sigma\nabla^1_\Sigma$,
normal velocity $v\None$ gives
\[
 \dot u=-e_0v,\qquad
 \dot \Hone=-\Delta^1_\Sigma v-(|\Aone|^2+\Ric^1(\None,\None))v.
\]
The meridional variation, with both moving $e_0$ slots retained, is
\begin{equation}\label{eq:dota}
 \dot{\Amer}=-e_0(e_0v)+({\sin u}/{\cos u})\ip{\grad u}{\grad v}
                   -v(|\Amer|^2+|\grad u|^2).
\end{equation}

The local form identity is
\begin{align}
 (\Qoneeps-\Qonehor)(v,w)=\int
 \biggl\{&\frac{\mu_\eps-1}{\cos u}\ip{\grad v}{\grad w}
       +\frac{\eps^2}{\Phi_\eps^3}(e_0v)(e_0w)\notag\\
 &-\frac{\mu_\eps-1}{\cos u}(|\Aone|^2+\Ric^1(\None,\None))vw\notag\\
 &-\left(\frac{\eps^2}{\Phi_\eps^3}-\frac{\mu_\eps-1}{\cos u}\right)
            (|\Amer|^2+|\grad u|^2)vw\biggr\}\dd\sigma_1 .\label{eq:common}
\end{align}
Differentiating \eqref{eq:He} and integrating by parts gives
\eqref{eq:common}; contraction with \eqref{eq:coeff} gives
\eqref{eq:star} (Appendix~\ref{app:quantitative-algebra}).
The full stationary normal Hessian, including zeroth-order
and prescribed-curvature terms, is given in
Appendix~\ref{app:graph-hessian}.
\end{proof}

Define the meridional density and its total mass by
\[
 m_\eps=4\eps^2\mu_\eps^3|\Amer|^2,\qquad
 \mathcal M_\eps=\int_{\Sigma^*} m_\eps\dd\sigma_H.
\]
Since $\mu_\eps\le1$, $d\sigma_H={\cos u}\,d\sigma_1$ and
$|\Amer|\le|\Aone|$, \eqref{sel:tag-3-4} gives
\[
 \mathcal M_\eps\le4\eps^2\int_\Sigma|\Aone|^2\dd\sigma_1<\infty.
\]

Using $\Phi_\eps\ge\cos u>0$, rationalization gives
\begin{equation}\label{eq:uniform-factors}
 0\le {\cos^{2} u}(1-\mu_\eps)
 =\frac{\eps^2{\cos^{2} u}{\sin^{2} u}}{\Phi_\eps(\Phi_\eps+{\cos u})}
 \le\tfrac12\eps^2{\sin^{2} u}.
\end{equation}
Equation~\eqref{eq:uniform-factors} bounds the first summand of
$j_\eps$ in \eqref{eq:star} from above by
$C\eps^2(1+|\Aone_{HH}|^2+|\AngleError|^2)$:
indeed, $|\mathsf E_{u,n}|\le|\AngleError|$ and
$4|\mathsf E_{u,n}|\le2+2|\AngleError|^2$.
Retaining the negative square gives the pointwise bound
\begin{equation}\label{eq:QRbound}
 j_\eps\le-m_\eps
       +C\eps^2(1+|\Aone_{HH}|^2+|\AngleError|^2).
\end{equation}

\subsection{The global stationary identity and master inequality}

The inequality $\En\le\mathcal R$ requires a global identity, whereas
\eqref{eq:local-noncmc-balance} and \eqref{eq:star} are local.
The factor $(\Hhor-2n)a_\omega$ in the centre identity need not be
integrable by itself near characteristic points. Combining it with
the Jacobi and central-variation terms under one cutoff gives
\eqref{eq:stationary-contraction} after passage to the limit.
Proposition~\ref{sel:prop:physical-form} controls the cutoff gradient
and flux terms.

With the same $\omega$, set $X_\omega=2(n-1)\omega T$.
On $\Sigma^*$ its horizontal speed $v^H_{X_\omega}=2(n-1)\alpha\omega$ gives the exact cancellation
\begin{equation}\label{eq:hat}
 a_\omega+v^H_{X_\omega}=2{\cos^{2} u}(2-\mathsf E_{u,n}).
\end{equation}
As a function of $N^1_T=\sin u$, $\omega$ has derivative
$d\omega/dN^1_T=2\cos u\to0$ and values $\omega\to\pm\pi/2$ as
$N^1_T\to\pm1$. Its constant extensions make $X_\omega$ a $C^1$ field
across the characteristic set.

\begin{lemma}[First variation along central fields]
\label{lem:central-flux}
\ResultBookmark{3}{lem:central-flux}{Lemma}%
Recall $\mu_\eps=(1+\eps^2\alpha^2)^{-1/2}$.
For a $C^1$ scalar surface function $\vartheta$ supported in a compact
noncharacteristic patch of $\Sigma_j$, let $\vartheta T$ denote any $C^1$
ambient extension of that surface field. Then
\begin{equation}\label{eq:central-noether}
 \delta P_\eps[\vartheta T]=\int\mu_\eps\,e_0\vartheta\dd\sigma_1,
 \qquad
 \delta(\Delta P_\eps)[\vartheta T]=\int(\mu_\eps-1)e_0\vartheta\dd\sigma_1 .
\end{equation}
\end{lemma}
\begin{proof}
Since $T$ is a Killing field preserving the anisotropic density
$\Phi_\eps$, only derivatives of $\vartheta$ contribute. Using
$T^\top={\cos u}e_0$ and the one-homogeneous normal gradient, evaluated at $\nu=\None$,
$D_\nu\Phi_\eps=(-{\cos u}e_{2n}+\eps^2{\sin u}T)/\Phi_\eps$, the tangential stress is
\[
 \Phi_\eps T^\top-{\sin u}(D_\nu\Phi_\eps)^\top
 =\frac{{\cos u}}{\Phi_\eps}
   \bigl(\Phi_\eps^2-(\eps^2-1){\sin^{2} u}\bigr)e_0
 =\mu_\eps e_0.
\]
At $\eps=0$ the stress is
$e_0$, proving \eqref{eq:central-noether}. Approximation of ambient
extensions in $C^1$ on a fixed noncharacteristic support gives the
stated $C^1$ case.
\end{proof}

The identities
$e_n\alpha=(\cos u)^{-2}e_nu$ and $\Aone_{HH}=\cos u\Shape$ give
\begin{equation}\label{eq:bounded-potential}
 q_H {\cos^{2} u}=|\Aone_{HH}|^2+4\langle \grad u,e_n\rangle+2(n+1){\sin^{2} u}.
\end{equation}
Also,
\begin{equation}\label{eq:trace-flux-domination}
 {\cos u}|Z_\omega|\le C(1+|\Aone|),
\end{equation}
since $\omega$ is bounded,
${\cos u}\ShapeError=\Aone_{HH}-{\cos u}\Amodel$ and
${\cos u}(\Hhor-2n)=\operatorname{tr}\Aone_{HH}-2n{\cos u}$.

On a compact noncharacteristic support, the curvature linearization
\cite[Lemma~2.8(i)]{JacobiSpectrum} is
\[
\dot \Hhor[X]=-\divh\grad v_X^H-q_Hv_X^H
                         +(X+v_X^H e_{2n})\Hhor,
 \qquad v_X^H=\langle X,-e_{2n}+\alpha T\rangle.
\]
For $X=v\None$, \eqref{eq:angle} gives $v_X^H=v/{\cos u}$ and
\[
 X+v_X^H e_{2n}=v\left(\frac{{\sin^{2} u}}{{\cos u}}e_{2n}+{\sin u}T\right)
                   =\frac{v{\sin u}}{{\cos u}}e_0.
\]
In $\Qonehor(v,w)=\int w\dot \Hhor[v\None]\dd\sigma_1$,
weighted horizontal integration by parts gives
\begin{equation}\label{eq:normal-horizontal-comparison}
 \Qonehor(v,w)=\Qhor(v/{\cos u},w/{\cos u})
                  +\int(vw/{\cos^{2} u}){\sin u}\,e_0\Hhor \dd\sigma_H.
\end{equation}
Write $\mathsf E_H:=\Hhor-\Heps=\Hhor-\Hpr_j$ on $\Sigma$.
The combined density is
\begin{equation}\label{eq:joint-curvature-density}
\begin{aligned}
 \mathfrak r_\eps={}&j_\eps
  +\mathsf E_H\bigl[\cos^2u\sin^2u\Hhor
       +\cos u(3-4\cos^2u)\Amer-2\cos^2u(2-\mathsf E_{u,n})\bigr]\\
 &+4(n-1)\cos u(1-\mu_\eps)\Amer
       +\cos^2u\sin u\,e_0\Hpr_j-(\Hpr_j-2n)a_\omega.
\end{aligned}
\end{equation}

\begin{lemma}[Curvature bounds for the residual densities]
\label{lem:finite-density-domination}
\ResultBookmark{3}{lem:finite-density-domination}{Lemma}%
Fix an index under the hypotheses of Proposition~\ref{sel:prop:physical-form}.
For the densities $j^H$, $j_\eps$ and
$\mathfrak r_\eps$ defined above,
\begin{equation}\label{eq:density-domination}
 {\cos u}\bigl(|j^H|+|\grad u|^2+\mathfrak e+|j_\eps|+|\mathfrak r_\eps|\bigr)
 \le C(1+|\Aone|^2).
\end{equation}
Consequently, all five densities lie in $L^1(d\sigma_H)$ by
\eqref{sel:tag-3-4}.
\end{lemma}
\begin{proof}
Equations~\eqref{eq:basic-curvature-domination} and
\eqref{eq:bounded-potential} control $j^H$, $\grad u$ and $\mathfrak e$.
The bounds ${\cos^{2} u}(1-\mu_\eps)\le1$ and
$\eps^2\mu_\eps^3\le1$ control $j_\eps$ by
$C(1+|\Aone|^2)$.
For $\mathfrak r_\eps$, use
\[
 \begin{aligned}
 |\cos u\,\mathsf E_H|
 &=|\cos u\,\Hhor-\cos u\,\Hpr_j|\\
 &=|\operatorname{tr}\Aone_{HH}-\cos u\,\Hpr_j|
 \le C(1+|\Aone_{HH}|).
 \end{aligned}
\]
The remaining coefficients satisfy
\[
\begin{gathered}
 |\cos^2u\sin^2u\Hhor+\cos u(3-4\cos^2u)\Amer|
 \le C\cos u(|\Aone_{HH}|+|\Amer|),\\
 |2\cos^2u(2-\mathsf E_{u,n})|
 \le C\cos^2u(1+|\mathsf E_{u,n}|).
\end{gathered}
\]
Finally,
\[
 \begin{aligned}
 |4(n-1)\cos u(1-\mu_\eps)\Amer|
 &\le C\cos u|\Amer|,\\
 |\cos^2u\sin u\,e_0\Hpr_j|
 &\le C\|\Hpr_j-2n\|_{C^1},\\
 \cos u|(\Hpr_j-2n)a_\omega|
 &\le C\|\Hpr_j-2n\|_\infty(1+|\mathsf E_{u,n}|),
 \end{aligned}
\]
where the last bound uses $\cos u|\alpha\omega|=|\sin u\,\omega|\le C$.
These bounds prove \eqref{eq:density-domination}.
\end{proof}

\begin{lemma}[Combined stationary identity]
\label{lem:global-identities}
\ResultBookmark{3}{lem:global-identities}{Lemma}%
Under the hypotheses of Proposition~\ref{sel:prop:physical-form},
\begin{equation}\label{eq:stationary-contraction}
 [\BC]_*=-\En-\int|\grad u|^2\dd\sigma_H-2n P_0(\Omega)
                 +\int\mathfrak r_\eps\dd\sigma_H-\int(\partial_{\None}\Hpr_j)\cos^3u\dd\sigma_H,
\end{equation}
\end{lemma}
\begin{proof}
Fix $j$ and $0\le\eta\le1$ in $C_c^1(\Sigma^*)$;
use $\eta^2$ in the linear identities. Contract \eqref{eq:common} and
\eqref{eq:normal-horizontal-comparison}, and add the local centre
identity \eqref{eq:local-noncmc-balance}. Since
$j^H=-\mathfrak e-|\grad u|^2-2n
       -(\Hhor-2n)a_\omega+\divh Z_\omega$,
the remaining curvature terms are
$-(\Hhor-2n)a_\omega+{\cos^{2} u}{\sin u}\,e_0\Hhor$.

Using $e_0u=\Amer$ and
$\operatorname{div}^1_\Sigma e_0=-{\sin u}\Hhor$,
the vector field $\cos^2u\sin u\,e_0$ has $d\sigma_H$-divergence
\[
 \begin{aligned}
 \operatorname{div}_{d\sigma_H}(\cos^2u\sin u\,e_0)
 &=(\cos u)^{-1}\operatorname{div}^1_\Sigma(\cos^3u\sin u\,e_0)\\
 &={\cos u}(4{\cos^{2} u}-3)\Amer-{\cos^{2} u}{\sin^{2} u}\Hhor\\
 &=-(\cos^2u\sin^2u \Hhor+\cos u(3-4\cos^2u) \Amer).
 \end{aligned}
\]
As $\Hhor=\Hpr_j+\mathsf E_H$ and $\Hhor-2n=\mathsf E_H+(\Hpr_j-2n)$, integration by
parts gives
\[
 \begin{aligned}
 \int\eta^2 \cos^2u\sin u\,e_0\mathsf E_H\dd\sigma_H
 &=-\int \mathsf E_H\operatorname{div}_{d\sigma_H}(\eta^2\cos^2u\sin u\,e_0)\dd\sigma_H\\
 &=\int\eta^2\mathsf E_H(\cos^2u\sin^2u\Hhor+\cos u(3-4\cos^2u)\Amer)\dd\sigma_H\\
 &\qquad-\int\mathsf E_H\cos^2u\sin u\,e_0(\eta^2)\dd\sigma_H.
 \end{aligned}
\]
Also $-[\mathsf E_H+(\Hpr_j-2n)]a_\omega
      =-2\mathsf E_H\cos^2u(2-\mathsf E_{u,n})+\mathsf E_H v^H_{X_\omega}-(\Hpr_j-2n)a_\omega$.
Apply \eqref{eq:central-noether} on the same support to
$\eta^2X_\omega$ for the term $\mathsf E_Hv^H_{X_\omega}$.
The first variation and the definition of $\mathsf E_H$ give
\[
 \begin{aligned}
 -\int\eta^2 \mathsf E_H v^H_{X_\omega}\dd\sigma_H
 &=\delta(\Delta P_\eps)[\eta^2 X_\omega]
   =2(n-1)\int(\mu_\eps-1)e_0(\eta^2\omega)\dd\sigma_1\\
 &=2(n-1)\int(\mu_\eps-1)
                 (2\eta^2 {\cos^{2} u}\Amer+\omega e_0(\eta^2))\dd\sigma_1.
 \end{aligned}
\]
Using $d\sigma_H={\cos u}\,d\sigma_1$, the term
$+\int\eta^2\mathsf E_Hv^H_{X_\omega}\dd\sigma_H$ in the contraction equals
\[
 \int\eta^2 4(n-1)\cos u(1-\mu_\eps) \Amer\dd\sigma_H
       -2(n-1)\int(\mu_\eps-1)\omega e_0(\eta^2)\dd\sigma_1.
\]
These substitutions give exactly $\mathfrak r_\eps$ in
\eqref{eq:joint-curvature-density} before any cutoff limit.
Write $R_\eta$ for the sum of the cutoff remainders in this proof:
\[
 [\BC]_{*,\eta}
 =\int\eta^2\{-\mathfrak e-|\grad u|^2-2n
               +\mathfrak r_\eps-(\partial_{\None}\Hpr_j){\cos^{3} u}\}\dd\sigma_H
       +R_\eta.
\]

For the localization remainders, expand the principal terms of
\eqref{eq:common} directly. The product rule gives
$\nabla^1_\Sigma(\eta v_i^1)=v_i^1\nabla^1_\Sigma\eta
 +\eta \partial_u v_i^1\nabla^1_\Sigma u$. Thus the difference between the
contracted principal density and $\eta^2$ times its value at $\eta=1$ is
\[
 \begin{aligned}
 &\frac{\mu_\eps-1}{\cos u}
  \bigl([v^1v^1]_*|\grad\eta|^2
       +2\eta[v^1\partial_uv^1]_*\langle\grad u,\grad\eta\rangle\bigr)\\
 &\qquad+\frac{\eps^2}{\Phi_\eps^3}
  \bigl([v^1v^1]_*(e_0\eta)^2+2\eta[v^1\partial_uv^1]_*\Amer\,e_0\eta\bigr).
 \end{aligned}
\]
For the scalar form, the same expansion with
$v_i^H=v_i^1/{\cos u}$ gives
\[
 [v^Hv^H]_*|\grad\eta|^2
 +2\eta[v^H\partial_uv^H]_*\langle \grad u,\grad\eta\rangle.
\]
Substituting \eqref{eq:coeff} and the contractions of the horizontal speeds, then using
$d\sigma_H={\cos u}\,d\sigma_1$, gives
\begin{equation}\label{eq:canonical-localization}
\begin{aligned}
 \int\bigl\{
 &\mu_\eps{\cos^{3} u}|\grad\eta|^2
 -2{\sin u}{\cos^{2} u}(2\mu_\eps-1)\eta
     \langle\grad u,\grad\eta\rangle\\
 &+\eps^2\mu_\eps^3\bigl({\cos u}(e_0\eta)^2
                  -4{\sin u}\eta \Amer\,e_0\eta\bigr)\bigr\}\dd\sigma_1.
\end{aligned}
\end{equation}
Here $0<\mu_\eps\le1$, $|2\mu_\eps-1|\le1$, and
${\cos u}\eps^2/\Phi_\eps^3=\mu_\eps\eps^2/\Phi_\eps^2\le1$ because $\Phi_\eps\ge\eps$.
Using
$|\nabla^1_\Sigma u|\le1+|\Aone|$, these terms are bounded by
$C\int\{{\cos^{3} u}|\nabla^1_\Sigma\eta|^2+
 {\cos^{2} u}(1+|\Aone|)|\nabla^1_\Sigma\eta|\}\dd\sigma_1$.
The remaining fluxes are those of $Z_\omega$, $\mathsf E_H\cos^2u\sin u\,e_0$ and the central
stress above. By \eqref{eq:trace-flux-domination}, the bound for ${\cos u}\mathsf E_H$
in Lemma~\ref{lem:finite-density-domination}, and boundedness of
$(\mu_\eps-1)\omega$, all remainders satisfy
\begin{equation}\label{eq:localization-bound}
 |R_\eta|\le C\int
 \{{\cos^{3} u}|\nabla^1_\Sigma\eta|^2
       +(1+|\Aone|)|\nabla^1_\Sigma\eta|\}\dd\sigma_1.
\end{equation}

Take $\eta=\eta_{j,\ell}$. Since ${\cos^{3} u}\le1$, \textup{(H2)}
controls the squared-gradient term in \eqref{eq:localization-bound};
\eqref{eq:cap-est} sends its linear-gradient term to zero.
Dominated convergence applies to the combined density by
Lemma~\ref{lem:finite-density-domination}, while
\cref{sel:prop:physical-form,lem:global-domain} give
$[\BC]_{*,\eta_{j,\ell}}\to[\BC]_*$.
This proves \eqref{eq:stationary-contraction}.
\end{proof}

\phantomsection\label{sec:trace}

The global identity \eqref{eq:stationary-contraction} still contains
the unconstrained Hessian $[\BC]_*$, since $I_F$ need not vanish.
Inserting $\Xvp=X_F-\kappa\Dil$ supplies the admissible Hessian;
the correction leaves dilation and central-translation terms.
The flux identity of Lemma~\ref{lem:global-domain} cancels their
constant and linear angle contributions; Cauchy--Schwarz controls the
remaining scalar square. Stability then gives $0\le\En\le\mathcal R$
in \eqref{eq:master}.

The combined residual is
\begin{align}
 \mathcal R={}&\int\mathfrak r_\eps\dd\sigma_H-\int(\partial_{\None}\Hpr_j)\cos^3u\dd\sigma_H
       +\tfrac{\pi^2}{4}\BC(T,T)+\BC(T,X_K)\notag\\
 &-2\kappa \mathcal R_{\Dil}(X_F)+\kappa^2\mathcal R_{\Dil}(\Dil)
       -\frac{2n\mathsf D\,I_F^2}{I_{\Dil}(2n+1)P_0(\Omega)}.
 \label{eq:Rexpand}
\end{align}
By the definition of the contraction,
$\BC(X_F,X_F)=[\BC]_*+\tfrac{\pi^2}{4}\BC(T,T)+\BC(T,X_K)$,
and bilinearity with $\Xvp=X_F-\kappa\Dil$ and $\kappa=I_F/I_{\Dil}$ gives
\[
 \BC(\Xvp,\Xvp)=\BC(X_F,X_F)-2\kappa \mathcal R_{\Dil}(X_F)+\kappa^2\mathcal R_{\Dil}(\Dil)
                         +2n I_F^2/I_{\Dil}.
\]
Since $\grad u=\AngleError-e_n$,
$\int|\grad u|^2\dd\sigma_H=
\int|\AngleError|^2\dd\sigma_H-2\beta+P_0(\Omega)$.
Substituting $\mathsf D=2n I_{\Dil}-(2n+1)P_0(\Omega)$ and
$2n I_F=(2n+1)P_0(\Omega)-\beta$ leaves the scalar cancellation
\[
 \begin{aligned}
 -2\beta+(2n+1)P_0(\Omega)
  -\frac{((2n+1)P_0(\Omega)-\beta)^2}{(2n+1)P_0(\Omega)}
  =-\frac{\beta^2}{(2n+1)P_0(\Omega)}.
 \end{aligned}
\]
Together with \eqref{eq:stationary-contraction}, these identities give
\begin{equation}\label{eq:master}
 \begin{aligned}
 \mathcal R-\En
 &=\BC(\Xvp,\Xvp)+\int|\AngleError|^2\dd\sigma_H
                  -\frac{\beta^2}{(2n+1)P_0(\Omega)}\\
 &\ge\BC(\Xvp,\Xvp)+\frac{2n}{2n+1}
                       \int|\AngleError|^2\dd\sigma_H\ge0.
 \end{aligned}
\end{equation}
Here $\beta^2\le P_0(\Omega)\int|\AngleError|^2\dd\sigma_H$
by Cauchy--Schwarz, and $\BC(\Xvp,\Xvp)\ge0$ by
Lemma~\ref{lem:global-domain}.
\subsection{The combined density estimate for
\texorpdfstring{$\mathcal R_1$}{R1}}\label{sec:joint-density}

By \textup{(H1)} and \eqref{eq:euler-assumption},
$\Heps=\Hpr_j$ and $|\Heps|\le2n+1$.
The mean-curvature identity \eqref{eq:He} gives
$\Amer=\Phi_\eps^3[\Heps(1-\mu_\eps)-\mu_\eps \mathsf E_H]/\eps^2$.
Substitute it into the $\mathsf E_H$- and $\Amer$-terms of
\eqref{eq:joint-curvature-density} and subtract $m_\eps/2$.
On $\Sigma^*$ set
\[
 a_\eps(s)=\frac{\Phi_\eps^3}{\eps^2}
       [\Heps(1-\mu_\eps)-\mu_\eps s]
\]
and let $p_\eps(s)$ be the resulting polynomial when
$\mathsf E_H$ is replaced by $s$:
\[
 \begin{aligned}
 p_\eps(s)={}&s\bigl[\cos^2u\sin^2u(\Heps+s)
    +\cos u(3-4\cos^2u)a_\eps(s)
    -2\cos^2u(2-\mathsf E_{u,n})\bigr]\\
 &+4(n-1)\cos u(1-\mu_\eps)a_\eps(s)
    -2\eps^2\mu_\eps^3a_\eps(s)^2.
 \end{aligned}
\]
The coefficient calculation in Appendix~\ref{app:quantitative-algebra}
gives
\[
 \frac12\partial_s^2p_\eps\le
 -\frac{\cos^2u\Phi_\eps^2}{\eps^2},\qquad
 |\partial_sp_\eps(0)|\le C\cos u\Phi_\eps(1+|\mathsf E_{u,n}|),\qquad
 p_\eps(0)\le C\eps^2.
\]
Since $\cos u>0$ on $\Sigma^*$, completing the square and setting
$s=\mathsf E_H$ gives
\begin{equation}\label{eq:joint-global-square}
 p_\eps(\mathsf E_H)\le C\eps^2(1+\mathsf E_{u,n}^2).
\end{equation}
Since $\Aone_{HH}={\cos u}(\Amodel+\ShapeError)$ and
$|\mathsf E_{u,n}|\le|\AngleError|$, \eqref{eq:defect-square} gives
\[
 |\Aone_{HH}|^2+|\AngleError|^2+\mathsf E_{u,n}^2\le C(1+\mathfrak e).
\]
Combining \eqref{eq:joint-global-square} with \eqref{eq:QRbound}
leaves half the negative square:
\begin{equation}\label{eq:joint-residual-pointwise}
 j_\eps+p_\eps(\mathsf E_H)+\tfrac12m_\eps
 \le-\tfrac12m_\eps+C\eps^2(1+\mathfrak e).
\end{equation}
The density \eqref{eq:joint-curvature-density} is precisely the
left side of \eqref{eq:joint-residual-pointwise} plus
$\cos^2u\sin u\,e_0\Hpr_j-(\Hpr_j-2n)a_\omega$.

The prescribed-curvature term satisfies
\[
 \left|\int (\Hpr_j-2n)a_\omega\dd\sigma_H\right|
 \le C\|\Hpr_j-2n\|_\infty\left(P_0(\Omega)+I_F+\sqrt{P_0(\Omega)\En}\right)
 \le C\|\Hpr_j-2n\|_\infty(V+\En).
\]
This uses $|a_\omega|\le C(1+|\mathsf E_{u,n}|+|\alpha |)$ and the bounds for $I_F$
in Lemma~\ref{lem:global-domain}.
Also,
\[
\left|\int {\cos^{2} u}{\sin u}\,e_0\Hpr_j\dd\sigma_H\right|
\le C\|\Hpr_j-2n\|_{C^1}P_0(\Omega).
\]
Integrate \eqref{eq:joint-residual-pointwise} using the integrability
proved in Lemma~\ref{lem:finite-density-domination}.
Together with \eqref{eq:joint-curvature-density}, this yields
\begin{equation}\label{eq:joint-residual-bound}
 \int\mathfrak r_\eps\dd\sigma_H
 \le-\tfrac12\mathcal M_\eps
       +C\eps^2(P_0(\Omega)+\En)+C\|\Hpr_j-2n\|_{C^1}(V+\En).
\end{equation}

\subsection{The mixed dilation estimate for
\texorpdfstring{$\mathcal R_2$}{R2}}
\label{sec:scaling}\label{app:potential-transport}

The dilation residual depends on the $\eps$-derivative of $P_\eps$.
The weighted cap integral is bounded by $\tfrac83\Delta P_\eps$
in \eqref{eq:scaling-excess-weight}; on $\{\cos u\ge1/4\}$,
the dilation-correction coefficients are $O(\eps^2)$.

Let $\Lambda_\eps=\eps\partial_\eps P_\eps$, with $\Omega$ fixed in
$\partial_\eps$. Differentiating
$P_\eps(\Dil_{e^s}\Omega)=e^{(2n+1)s}P_{\eps e^{-s}}(\Omega)$ at $s=0$ gives
\[
 \delta_{\Dil}P_\eps=(2n+1)P_\eps-\Lambda_\eps.
\]
For a smooth ambient field $Y$ whose restriction to $\Sigma$ is
compactly supported in $\Sigma^*$, differentiate the Gauss--Green identity
\[
 \delta_{\Dil}\left(P_\eps-\int_\Omega\Hpr_j\dd x\right)
 =(2n+1)P_\eps-\Lambda_\eps
       -\int_\Omega(\Dil\Hpr_j+(2n+2)\Hpr_j)\dd x
\]
along $Y$, holding $\Dil$ fixed. Stationarity removes the acceleration
term and gives $\delta_YP_\eps=\int\Hpr_jv_Y^H\dd\sigma_H$.
Adding $2n\delta_YV$ yields
\begin{equation}\label{eq:mixed}
 \mathcal R_{\Dil}(Y)=-\delta_Y\Lambda_\eps
       +\int(2n-\Hpr_j-\Dil\Hpr_j)v_Y^H\dd\sigma_H.
\end{equation}
Fix a smooth cutoff $0\le\zeta({\cos u})\le1$, equal to
one for ${\cos u}\le1/4$ and zero for ${\cos u}\ge1/2$.
Its derivative is supported in the noncharacteristic transition
$1/4\le{\cos u}\le1/2$.
By the construction of $X_F$ above, $X_F=(F/\alpha)T$ on the support of $\zeta$, where
$F/\alpha={\cos u}/{\sin u}+u$ and
$\partial_u(F/\alpha)=-{\cos^{2} u}/{\sin^{2} u}$.
Here $|{\sin u}|\ge\sqrt3/2$, so the quotient is smooth.
Holding the surface, field and $\gTheta$ fixed, differentiate
\eqref{eq:central-noether} in $\eps$:
\[
 \eps\partial_\eps\mu_\eps
 =\eps\partial_\eps\frac {\cos u}\Phi_\eps
 =-\frac{{\cos u}\eps^2{\sin^{2} u}}{\Phi_\eps^3}.
\]
Since $e_0(F/\alpha)=\partial_u(F/\alpha)\Amer=-{\cos^{2} u}\Amer/{\sin^{2} u}$ and $(F/\alpha)={\cos u}F/{\sin u}$,
\[
 \begin{aligned}
 \delta \Lambda_\eps[\zeta (F/\alpha)T]
 &=-\eps^2\int\frac{{\cos u}{\sin^{2} u}}{\Phi_\eps^3}
                   \{\zeta e_0(F/\alpha)+(F/\alpha)e_0\zeta\}\dd\sigma_1\\
 &=\eps^2\int\frac{\zeta{\cos^{3} u}\Amer
                  -{\cos^{2} u}{\sin u}F e_0\zeta}{\Phi_\eps^3}\dd\sigma_1\\
 &=\eps^2\int\frac{\zeta{\cos^{2} u}\Amer
                  -{\sin u}{\cos u}F e_0\zeta}{\Phi_\eps^3}\dd\sigma_H.
 \end{aligned}
\]
Off the cap, differentiating \eqref{eq:He} in $\eps$ and using
$1-\mu_\eps^2=\eps^2\sin^2u/\Phi_\eps^2$ gives
$\eps\partial_\eps \Heps
 =\frac{\eps^2}{\Phi_\eps^3}\{-{\cos u}{\sin^{2} u}\Hhor
                       +(3\mu_\eps^2-1)\Amer\}$.
Thus
\begin{equation}\label{eq:scaling-full-flux}
\begin{aligned}
 \delta \Lambda_\eps[X_F]
 ={}&\eps^2\int\frac{\zeta{\cos^{2} u}\Amer
                 -{\sin u}{\cos u}F e_0\zeta}{\Phi_\eps^3}\dd\sigma_H\\
 &+\eps^2\int(1-\zeta)\frac{F}{\Phi_\eps^3}
     \{-{\cos u}{\sin^{2} u}\Hhor+(3\mu_\eps^2-1)\Amer\}\dd\sigma_H.
\end{aligned}
\end{equation}

For $\eps>0$, $\Lambda_\eps$ has smooth anisotropic density
$\eps^2{\sin^{2} u}/\Phi_\eps$ on the $\gTheta$ unit normal bundle.
Its stress is bounded, and its first-variation curvature density is
bounded by $C_j(1+|\Aone|)$.

For a unit tangent $E$, differentiating $\zeta X_F=(\zeta F/\alpha)T$
and $e_{2n}=-(\None-\sin u\,T)/\cos u$ gives
\[
 \begin{aligned}
 \nabla^1_E(\zeta X_F)
 &=\partial_u(\zeta F/\alpha)E(u)T+(\zeta F/\alpha)\nabla^1_ET,\\
 \nabla^1_E e_{2n}&=-\frac{\Aone E}{\cos u}
       +E(u)T+\tan u\bigl(\nabla^1_ET+E(u)e_{2n}\bigr).
 \end{aligned}
\]
The coefficients are bounded on the cap and on $\cos u\ge1/4$,
respectively. Since $|E(u)|\le1+|\Aone|$ and $|\nabla^1_ET|\le1$, both
$\zeta X_F$ and $(1-\zeta)X_F$ are bounded and their surface
derivatives are bounded by $C_j(1+|\Aone|)$.
The common-cutoff stress cost is at most
$C_j\int|\nabla^1_\Sigma(\eta_\ell^2)|\dd\sigma_1\to0$ by
\eqref{eq:cap-est}; all remaining densities are integrable by \eqref{sel:tag-3-4}.
Thus \eqref{eq:scaling-full-flux} holds globally.
Proposition~\ref{sel:prop:physical-form} also passes \eqref{eq:mixed}
to $Y=X_F$, whose horizontal speed is $F$.

On $0<{\cos u}\le1/2$, rationalizing with
$\Phi_\eps^2-\cos^2u=\eps^2\sin^2u$ gives
\[
 \frac{\eps^2\cos u/\Phi_\eps^3}{\Phi_\eps/\cos u-1}
 =\frac{\mu_\eps^2(1+\mu_\eps)}{\sin^2u}\le\frac83,
\]
since $\mu_\eps\le1$ and $\sin^2u\ge3/4$.
By \textup{(H2)}, $\Sigma^*$ carries all the Riemannian perimeter.
Multiplication by the excess density and integration give
\begin{equation}\label{eq:scaling-excess-weight}
 \begin{gathered}
 \Delta P_\eps
 =\int_{\Sigma^*}(\Phi_\eps-\cos u)\dd\sigma_1
 =\int_{\Sigma^*}(\Phi_\eps/\cos u-1)\dd\sigma_H,\\
 \int_{\{{\cos u}\le1/2\}}\frac{\eps^2{\cos u}}{\Phi_\eps^3}\dd\sigma_H
 \le\frac83\Delta P_\eps.
 \end{gathered}
\end{equation}

On ${\cos u}\ge1/4$, $\Phi_\eps^{-3}\le64$ and
$F,\zeta'$ are bounded. Since
$e_0\zeta=-{\sin u}\zeta'({\cos u})\Amer$,
the coefficients of $\Amer$ and $\Hhor$ in
\eqref{eq:scaling-full-flux} are $O(\eps^2)$.
Here $\mu_\eps\ge\cos u\ge1/4$, so
$m_\eps\ge\eps^2|\Amer|^2/16$ and Young's inequality gives
\[
 C|\kappa|\eps^2\int_{\{\cos u\ge1/4\}}|\Amer|\dd\sigma_H
 \le\tfrac18\mathcal M_\eps+C\eps^2\kappa^2P_0(\Omega).
\]
On the cap, completing a square pointwise gives
\[
 2|\kappa|\frac{\eps^2\cos^2u}{\Phi_\eps^3}|\Amer|
 \le\tfrac14m_\eps+\kappa^2\frac{\eps^2\cos u}{\Phi_\eps^3}.
\]
Multiply by $0\le\zeta\le1$ and apply
\eqref{eq:scaling-excess-weight}. The cap contribution is at most
$\mathcal M_\eps/4+(8/3)\kappa^2\Delta P_\eps$.
Their combined cost is at most $3\mathcal M_\eps/8$, including
where the supports overlap.
For the curvature term, ${\cos u}\Hhor=\operatorname{tr}\Aone_{HH}$,
\eqref{eq:bound} and \eqref{eq:global-support} give
\[
 \begin{aligned}
 C|\kappa|\eps^2\int_{\{\cos u\ge1/4\}}|\Hhor|\dd\sigma_H
 &\le C|\kappa|\eps^2P_0(\Omega)^{1/2}
                  (P_0(\Omega)+\En)^{1/2}\\
 &\le C\eps^2(V+\En).
 \end{aligned}
\]
Together with \eqref{eq:scaling-full-flux}, these bounds yield
\begin{equation}\label{eq:scaling-direct-bound}
 2|\kappa\,\delta\Lambda_\eps[X_F]|
 \le\tfrac38\mathcal M_\eps+\tfrac83\kappa^2\Delta P_\eps
       +C\eps^2(V+\En).
\end{equation}
Since $F>0$, the prescribed-curvature term in \eqref{eq:mixed}
is bounded in absolute value by
$C\|\Hpr_j-2n\|_{C^1}I_F$. Multiplying by $2|\kappa|$ and using
\eqref{eq:global-support} gives
\begin{equation}\label{eq:scaling-bound}
 -2\kappa\mathcal R_{\Dil}(X_F)
 \le\tfrac38\mathcal M_\eps+\tfrac83\kappa^2\Delta P_\eps
       +C(\eps^2+\|\Hpr_j-2n\|_{C^1})(V+\En).
\end{equation}
The positive excess term will be combined with the negative excess
in \eqref{eq:dilation-residual-combined}.

\subsection{The remaining residual
\texorpdfstring{$\mathcal R_3$}{R3} and absorption}\label{sec:remaining-residual}

\[
 \begin{aligned}
 \mathcal R_3={}&-\int(\partial_{\None}\Hpr_j)\cos^3u\dd\sigma_H
       +\tfrac{\pi^2}{4}\BC(T,T)+\BC(T,X_K)\\
 &+\kappa^2\mathcal R_{\Dil}(\Dil)
       -\frac{2n\mathsf D\,I_F^2}{I_{\Dil}(2n+1)P_0(\Omega)}.
 \end{aligned}
\]

Central translation preserves $P_\eps$ and $V$. Since
$\operatorname{div}T=0$, change of variables gives
$\delta_T\int_\Omega \Hpr_j\dd x=\int_\Omega T\Hpr_j\dd x$.
At a stationary set, differentiating this identity along $X_K$ and $T$
with $T$ held fixed, and applying Gauss--Green, gives
\begin{equation}\label{eq:central-pinning}
 \BC(T,X_K)=-\int(T\Hpr_j)K\dd\sigma_H,\qquad
 \BC(T,T)=-\int_\Omega T^2\Hpr_j\dd x.
\end{equation}
The local identity is $\BC(T,Y)=-\int(T\Hpr_j)v_Y^1\dd\sigma_1$
for $Y$ supported in a compact subset of $\Sigma^*$.
Apply it to $Y=\eta_{j,\ell}X_K$ and $Y=\eta_{j,\ell}T$.
The left sides converge by \cref{sel:prop:physical-form}; the right
sides converge by boundedness of the speeds and finite $\sigma_1$ area.
Gauss--Green gives
$\int(T\Hpr_j)v_T^1\dd\sigma_1=\int_\Omega T^2\Hpr_j\dd x$.
Since $K$ is bounded, the absolute
values are at most $C\|\Hpr_j-2n\|_{C^1}P_0(\Omega)$ and
$C|\Omega|\|\Hpr_j-2n\|_{C^2}$,
respectively.

The contracted prescribed-curvature correction is
$-\int(\partial_{\None}\Hpr_j)\cos^3u\dd\sigma_H$; its absolute value is bounded by
$C\|\Hpr_j-2n\|_{C^1}P_0(\Omega)$.

The twice-differentiated scaling identities and stationarity give
the following combined residual; the intermediate expansion is in
Appendix~\ref{app:quantitative-algebra}.
\begin{equation}\label{eq:dilation-residual-combined}
\begin{aligned}
 &\kappa^2\mathcal R_{\Dil}(\Dil)-\frac{2n\mathsf D\,I_F^2}{I_{\Dil}(2n+1) P_0(\Omega)}\\
 &=\kappa^2\left\{-(2n+1)\Delta P_\eps-2n\Lambda_\eps
       +\eps\partial_\eps \Lambda_\eps
       -\int_\Omega(\Dil^2\Hpr_j+(2n+2)\Dil\Hpr_j)\dd x
       -\frac{\mathsf D^2}{(2n+1) P_0(\Omega)}\right\}.
\end{aligned}
\end{equation}
On the fixed surface, differentiating the $d\sigma_1$-density of
$\Lambda_\eps$ gives
\[
 0\le\eps\partial_\eps\left(\eps^2\sin^2u/\Phi_\eps\right)
   =\left(2-\frac{\eps^2{\sin^{2} u}}{\Phi_\eps^2}\right)\left(\eps^2\sin^2u/\Phi_\eps\right)
   \le2\left(\eps^2\sin^2u/\Phi_\eps\right).
\]
Since $n\ge2$, $-2n\Lambda_\eps+\eps\partial_\eps\Lambda_\eps
\le-(2n-2)\Lambda_\eps\le0$. The prescribed-curvature integral in
\eqref{eq:dilation-residual-combined} has absolute value at most
$C\kappa^2V\|\Hpr_j-2n\|_{C^2}$. Using
\eqref{eq:global-support} and the central and normal-derivative
bounds above, while retaining the negative terms, gives
\begin{equation}\label{eq:other}
 \begin{aligned}
 \mathcal R_3\le{}&-(2n+1)\kappa^2\Delta P_\eps
       -(2n-2)\kappa^2\Lambda_\eps
       -\frac{\kappa^2\mathsf D^2}{(2n+1)P_0(\Omega)}\\
       &+C\|\Hpr_j-2n\|_{C^2}(V+\En).
 \end{aligned}
\end{equation}
\phantomsection\label{sec:absorption-centre}

\begin{proof}[Proof of Theorem~\ref{thm:absorption}]
Combine the master inequality \eqref{eq:master} with
\eqref{eq:joint-residual-bound}, \eqref{eq:scaling-bound} and
\eqref{eq:other}. Since $1/2-3/8=1/8$, this gives
\begin{equation}\label{eq:absorption}
 \begin{aligned}
 \En+\tfrac18\mathcal M_\eps
 &+(2n+1-\tfrac83)\kappa^2\Delta P_\eps
 +(2n-2)\kappa^2\Lambda_\eps
 +\frac{\kappa^2\mathsf D^2}{(2n+1)P_0(\Omega)}\\
 &\le C_0(\eps^2+\|\Hpr_j-2n\|_{C^2(U)})(V+\En).
 \end{aligned}
\end{equation}
All constants depend only on $n,\mathcal K,U$: the flux estimates
use $P_0(\Omega)/V\le C$ from stationarity, and the fixed dilation
cutoff and angular thresholds have bounded derivatives.
The fixed-$j$ curvature estimates are used only for the preceding
closure, before these uniform inequalities.

For a uniform positive lower bound on $\kappa$, put
$Z=\sum_k(x_kX_k+y_kY_k)$, so $\operatorname{div}Z=2n$.
Let $R=1+\sup_{\mathcal K}|z|$. Horizontal Gauss--Green gives
$2nV\le RP_0(\Omega)$, and $F\ge1$ gives
\[
 \kappa=\frac{I_F}{(2n+2)V}
 \ge\frac{P_0(\Omega)}{(2n+2)V}
 \ge\frac{2n}{(2n+2)R}>0.
\]
Choose $\delta_*>0$ with $C_0\delta_*\le1/2$.
Absorb the energy on the right of \eqref{eq:absorption} and use
$2n+1-8/3>0$. Restoring the index $j$ gives
\[
 \En_j+\mathcal M_{\eps_j}+\Delta P_j
 \le C(\eps_j^2+\|\Hpr_j-2n\|_{C^2(U)})|\Omega_j|.
\]
Equations~\eqref{eq:centre-derivative-bound} and \eqref{eq:coercive}
then imply \eqref{eq:full-centre-decay}.
\end{proof}

\Needspace{12\baselineskip}
\appendix
\section{Curvature energy and cutoff closure}
\label{app:admissible-variations}

At fixed $\eps=\eps_j>0$, the hypotheses of
Proposition~\ref{sel:prop:physical-form} imply
$\int(1+|\Aone|^2)\dd\sigma_1<\infty$.
The capacity cutoffs then close the Jacobi form on bounded speeds with
$|\nabla^1_\Sigma v|\le C_j(1+|\Aone|)$, including speeds whose
ambient fields move characteristic points. This supplies the finite
global pairings used in Sections~\ref{app:finite-master} and
\ref{sec:proof-main}.

\begin{proof}[Proof of Proposition~\ref{sel:prop:physical-form}]
Put $\eps=\eps_j$ and use $\Phi_\eps$ from \eqref{eq:anisotropy}.
Raising the $g_1$ unit conormal with $g_\eps$ gives
\begin{equation}\label{sel:eq:metric-normal}
 \Neps=\frac{N_H+\eps^2{\sin u}T}{\Phi_\eps},\qquad
 d\sigma_\eps=\eps^{-1}\Phi_\eps\,d\sigma_{1,j},
 \qquad v_X^\eps=v_X^1/\Phi_\eps.
\end{equation}
The volume-speed measure satisfies
\begin{equation}\label{sel:eq:metric-volume}
 \eps v_X^\eps\,d\sigma_\eps=v_X^1\,d\sigma_{1,j}.
\end{equation}
Here $\eps\le\Phi_\eps\le1$ and $\gTheta\le g_\eps\le\eps^{-2}\gTheta$.
For horizontal $X$, the connection difference is
\[
 (\nabla^\eps-\nabla^1)_XT
 =(\nabla^\eps-\nabla^1)_TX=(\eps^{-2}-1)JX,
\]
and it vanishes on horizontal pairs and on $(T,T)$.
The chain rule and $\cos u/\Phi_\eps\le1$ give
\[
 \nabla^1_\Sigma\Phi_\eps
 =\frac{(\eps^2-1)\sin u\cos u}{\Phi_\eps}\nabla^1_\Sigma u,
 \qquad |\nabla^1_\Sigma\Phi_\eps|\le1+|\Aone|.
\]
At fixed $\eps>0$, the normal map \eqref{sel:eq:metric-normal} and
its inverse have bounded derivatives on the unit sphere in the
left-invariant frame. Differentiation and the connection difference give,
with $C_j=C(n,\eps_j)\ge1$,
\begin{equation}\label{sel:eq:metric-comparison}
 \begin{gathered}
 d\sigma_{1,j}\le d\sigma_\eps\le\eps^{-1}d\sigma_{1,j},\qquad
 \eps|\nabla^1_{\Sigma_j}f|\le|\nabla^\eps_\Sigma f|
       \le|\nabla^1_{\Sigma_j}f|,\\
 C_j^{-1}(1+|\Aone_j|)\le1+|\Aeps|
       \le C_j(1+|\Aone_j|).
 \end{gathered}
\end{equation}
The product rule in both directions and \eqref{sel:eq:metric-comparison} give
\[
 \begin{gathered}
 \nabla^1_\Sigma v_X^\eps
 =\Phi_\eps^{-1}\nabla^1_\Sigma v_X^1
  -\Phi_\eps^{-2}v_X^1\nabla^1_\Sigma\Phi_\eps,\\
 C_j^{-1}\|v_X^1\|_j^2
 \le\int\bigl\{|\nabla^\eps_\Sigma v_X^\eps|^2
                 +(1+|\Aeps|^2)(v_X^\eps)^2\bigr\}\dd\sigma_\eps
 \le C_j\|v_X^1\|_j^2.
 \end{gathered}
\]

For a $g_\eps$-normal speed $v\in C_c^\infty(\Sigma_j^*)$,
the Riemannian second variation and \eqref{sel:eq:elliptic-metric} give
\begin{equation}\label{sel:eq:elliptic-metric-jacobi}
 \eps^{-1}\mathcal B^{\eps_j}_{1,j}(\Phi_\eps v,\Phi_\eps v)
 =\int\left\{|\nabla^\eps_\Sigma v|^2-
 (|\Aeps|^2+\Ric^\eps(\Neps,\Neps)
           +\partial_{\Neps}\Hpr_j)v^2\right\}\dd\sigma_\eps.
\end{equation}
This is the $f\equiv1$ formula in \cite[Theorem~2.8]{Verzellesi2026},
localized by \cite[Lemma~2.4]{Verzellesi2026}; stationarity cancels
the $\Heps-\Hpr_j$ acceleration terms.

Equation~\eqref{sel:eq:metric-comparison} gives finite $g_\eps$ area,
and $|\Ric^\eps(\Neps,\Neps)|
 +|\partial_{\Neps}\Hpr_j|\le C_j$.
Since $P_0(\Omega)>0$, choose $0\le\zeta\in C_c^\infty(\Sigma_j^*)$
with $\int\zeta\dd\sigma_{1,j}=1$.
For compactly supported $C^1$ speeds with $\int v\dd\sigma_\eps=0$,
smooth approximation on a fixed noncharacteristic support followed
by a correction with $\zeta$ preserves zero mean and passes stability.

For a cutoff $\eta_\ell$ from \textup{(H2)}, put
\[
 b_\ell=\frac{\int\eta_\ell\dd\sigma_\eps}{\int\zeta\dd\sigma_\eps},
 \qquad v_\ell=\eta_\ell-b_\ell\zeta.
\]
Since $\int\zeta\dd\sigma_\eps\ge1$, $b_\ell$ is bounded and
$\int v_\ell\dd\sigma_\eps=0$.
Stability in \eqref{sel:eq:elliptic-metric-jacobi},
$\eta_\ell^2\le2v_\ell^2+2b_\ell^2\zeta^2$, and bounded curvature on
$\spt\zeta$ give
\[
 \int\eta_\ell^2|\Aeps|^2\dd\sigma_\eps
 \le2\int|\nabla^\eps_\Sigma v_\ell|^2\dd\sigma_\eps+C_j
 \le C_j\left(1+\int|\nabla^\eps_\Sigma\eta_\ell|^2\dd\sigma_\eps\right).
\]
Equation~\eqref{sel:eq:metric-comparison}, \textup{(H2)} and Fatou give
$\int|\Aeps|^2\dd\sigma_\eps<\infty$, hence \eqref{sel:tag-3-4}.

For compactly supported $C^1$ $g_1$-normal speeds, Cauchy--Schwarz in
\eqref{sel:eq:elliptic-metric-jacobi} and the energy comparison give
$|\mathcal B^{\eps_j}_{1,j}(v,w)|\le C_j\|v\|_j\|w\|_j$.
For admissible $v,w$, \eqref{sel:tag-3-4} gives $\|v\|_j+\|w\|_j<\infty$.

Writing $\eta=\eta_{j,\ell}$, the product rule gives
\[
 \begin{aligned}
 \|\eta v-v\|_j^2
 &\le2\int(1-\eta)^2|\nabla^1_\Sigma v|^2\dd\sigma_{1,j}
       +2\|v\|_\infty^2\int|\nabla^1_\Sigma\eta|^2\dd\sigma_{1,j}\\
 &\quad+\int(1+|\Aone_j|^2)(1-\eta)^2v^2\dd\sigma_{1,j}\longrightarrow0.
 \end{aligned}
\]
The first and last terms vanish by dominated convergence,
the middle one by \textup{(H2)}.

\Needspace{5\baselineskip}
For two cutoffs $\eta,\widetilde\eta$, bilinearity gives
\[
 \begin{aligned}
 |\mathcal B^{\eps_j}_{1,j}(\eta v,\eta w)-\mathcal B^{\eps_j}_{1,j}(\widetilde\eta v,\widetilde\eta w)|
 \le C_j\bigl(&\|(\eta-\widetilde\eta)v\|_j\|\eta w\|_j\\
              &+\|\widetilde\eta v\|_j\|(\eta-\widetilde\eta)w\|_j\bigr).
 \end{aligned}
\]
This proves cutoff independence. Also
$\|(\eta-1)v\dd\sigma_{1,j}\|_{\rm TV}
=\int(1-\eta)|v|\dd\sigma_{1,j}\to0$
by dominated convergence.

If $\int v\dd\sigma_{1,j}=0$, the same $\zeta$ gives
\[
 \begin{gathered}
 v_\ell=\eta_{j,\ell}v
          -\left(\int\eta_{j,\ell}v\dd\sigma_{1,j}\right)\zeta,
 \qquad
 \int v_\ell\dd\sigma_{1,j}=0,
 \\
 \int\eta_{j,\ell}v\dd\sigma_{1,j}\to0,\qquad
 \|v_\ell-v\|_j\to0.
 \end{gathered}
\]
By \eqref{sel:eq:metric-volume}, $v_\ell/\Phi_\eps$ has zero
$d\sigma_\eps$ mean. Stability for compactly supported $C^1$ speeds gives
$\mathcal B^{\eps_j}_{1,j}(v_\ell,v_\ell)\ge0$;
continuity gives $\mathcal B^{\eps_j}_{1,j}(v,v)\ge0$.
\Needspace{8\baselineskip}
Finally, Cauchy--Schwarz gives
\[
 \int(1+|\Aone_j|)|\nabla^1_\Sigma\eta_{j,\ell}|\dd\sigma_{1,j}
 \le\left(2\int(1+|\Aone_j|^2)\dd\sigma_{1,j}\right)^{1/2}
       \left(\int|\nabla^1_\Sigma\eta_{j,\ell}|^2\dd\sigma_{1,j}\right)^{1/2}
 \longrightarrow0
\]
by \eqref{sel:tag-3-4} and \textup{(H2)}, proving \eqref{eq:cap-est}.

\phantomsection\label{eq:ambient-hessian-closure}
At fixed $j$, the ambient Hessian is continuous in the surface
$H^1$ norm of vector-field traces. Indeed, in Euclidean coordinates
the perimeter integrand is a smooth function of the base point and
the nonzero oriented conormal. The cofactor formula for
$x\mapsto x+sX(x)+rY(x)$ involves only the surface traces and their
first tangential derivatives; its coefficients are bounded on the
compact set for fixed $\eps_j>0$. For example, with Euclidean unit
normal $\nu$, tangent frame $E_i$ and $v=\langle X,\nu\rangle$,
the second derivative of the prescribed-volume term along
$x+sX(x)$ has boundary density
\[
 (D\Hpr_j\cdot X)v
 +\Hpr_j\left(v\operatorname{div}_\Sigma X
       -\sum_i\langle\nu,D_{E_i}X\rangle\langle X,E_i\rangle\right).
\]
Polarization gives the same first-derivative bound for mixed terms.
Stationarity removes the acceleration term when another ambient
extension or variation is used. Euclidean and $g_1$ surface norms
are equivalent on the fixed compact set. Consequently
\[
 |D^2(P_\eps-\textstyle\int_\Omega\Hpr_j)[X,Y]|
 \le C_j\|X\|_{H^1(\Sigma,g_1)}\|Y\|_{H^1(\Sigma,g_1)}.
\]
For bounded fields with
$|\nabla^1_\Sigma X|\le C_j(1+|\Aone|)$, the curvature bound and
\textup{(H2)} give $\eta_{j,\ell}X\to X$ in this vector-field norm.
On compact regular supports, stationarity identifies the ambient
Hessian with the normal Jacobi pairing. The vector-field convergence
and the scalar convergence in $\|\cdot\|_j$ therefore identify their
limits. In particular, this argument uses the weighted scalar norm
$\int(|\nabla^1_\Sigma v|^2+(1+|\Aone|^2)v^2)\dd\sigma_1$;
ordinary scalar $H^1$ continuity is not needed.
\end{proof}

\section{Algebraic details for the quantitative estimate}
\label{app:quantitative-algebra}

For \eqref{eq:dota}, use
$\grad v=\nabla^1_\Sigma v-(e_0v)e_0$, with $\grad v\perp e_0$. Then
\[
(\dot e_0)^\top= ve_n +({\sin u}/{\cos u})\grad v,\quad
 \nabla^{1,\Sigma}_{e_0}e_0=-({\sin u}/{\cos u})(\AngleError + e_n ),\quad
 \operatorname{Sec}^1(\None,e_0)=1.
\]
The sectional-curvature value follows because
$\operatorname{span}\{\None,e_0\}=\operatorname{span}\{e_{2n},T\}$ is a
horizontal--vertical plane for $\gTheta$.
Differentiating $\Amer=\ip{\Aone e_0}{e_0}$ and retaining both moving
$e_0$ slots gives
\[
 \dot{\Amer}=-\operatorname{Hess}^1_\Sigma v(e_0,e_0)
       -v(|\Aone e_0|^2+1)+2\langle \Aone e_0,(\dot e_0)^\top\rangle.
\]
Substitute $\Aone e_0=\AngleError+\Amer e_0$ from the block decomposition and
$\operatorname{Hess}^1_\Sigma v(e_0,e_0)
=e_0(e_0v)-\langle\nabla^{1,\Sigma}_{e_0}e_0,\nabla^1_\Sigma v\rangle$.
The gradient terms combine as
\[
 -\frac{\sin u}{\cos u}\ip{\AngleError+e_n}{\grad v}
 +2\frac{\sin u}{\cos u}\ip{\AngleError}{\grad v}
 =\frac{\sin u}{\cos u}\ip{\grad u}{\grad v}.
\]
The terms multiplying $v$ combine to
$-|\Amer|^2-|\AngleError|^2-1+2\langle\AngleError,e_n\rangle=-|\Amer|^2-|\grad u|^2$.
These terms give \eqref{eq:dota}.

At fixed $\eps$, differentiate the second identity in \eqref{eq:He},
insert $\dot \Hone$ and \eqref{eq:dota}, and integrate the two
second-derivative terms by parts. The terms containing
$\partial_u(\eps^2/\Phi_\eps^3)$ cancel.
Use $\nabla^1_\Sigma u=\grad u+\Amer e_0$ and
$\operatorname{div}^1_\Sigma e_0=-({\sin u}/{\cos u})(\Hone-\Amer)$.
The remaining first-order terms are
$w\bigl(\langle\grad u,\grad v\rangle+(\Amer-\Hone)e_0v\bigr)$
times the coefficient
\[
 \partial_u\left(\frac{\mu_\eps-1}{\cos u}\right)
 +\left(\frac{\eps^2}{\Phi_\eps^3}
             -\frac{\mu_\eps-1}{\cos u}\right)\tan u
 =\frac{\partial_u\mu_\eps+(\eps^2/\Phi_\eps^3)\sin u}{\cos u}=0.
\]
This proves \eqref{eq:common}.

For the three speeds,
$\nabla^1_\Sigma v_i^1=\partial_u v_i^1(\grad u+\Amer e_0)$.
By \eqref{eq:coeff},
$[\partial_uv^1\,\partial_uv^1]_*=-3{\cos^{4} u}$ contracts the principal term in
\eqref{eq:common} to
\[
 -3{\cos^{4} u}\left(\frac{\mu_\eps-1}{\cos u}|\grad u|^2
                         +\frac{\eps^2}{\Phi_\eps^3}|\Amer|^2\right),
\]
Adding the zeroth-order contraction, with $[v^1v^1]_*={\cos^{4} u}$, and using
$|\Aone|^2=|\Aone_{HH}|^2+2|\AngleError|^2+|\Amer|^2$
gives the $d\sigma_1$-density
\[
 \begin{aligned}
 &-{\cos^{4} u}\biggl\{\frac{\mu_\eps-1}{\cos u}
      (|\Aone_{HH}|^2+2|\AngleError|^2+2|\grad u|^2+\Ric^1(\None,\None))
        +\frac{\eps^2}{\Phi_\eps^3}(|\grad u|^2+4|\Amer|^2)\biggr\}.
 \end{aligned}
\]
Since $|\grad u|^2=|\AngleError|^2-2\mathsf E_{u,n}+1$,
\[
 \begin{aligned}
 2|\AngleError|^2+2|\grad u|^2+\Ric^1(\None,\None)
 &=4|\AngleError|^2-4\mathsf E_{u,n}+2+2n-2(n+1){\cos^{2} u}\\
 &=4|\AngleError|^2-4\mathsf E_{u,n}+2(n+1){\sin^{2} u}.
 \end{aligned}
\]
Finally, $d\sigma_1=(\cos u)^{-1}d\sigma_H$ converts the two
weights to $\cos^2u(1-\mu_\eps)$ and $-\eps^2\mu_\eps^3$,
giving \eqref{eq:star}.

The polynomial of Section~\ref{sec:joint-density} expands on
$\Sigma^*$ as
\[
\begin{aligned}
 p_\eps(s)={}&s\biggl\{
  \cos^2u\sin^2u(\Heps+s)-2\cos^2u(2-\mathsf E_{u,n})\\
 &\qquad+\frac{\cos u(3-4\cos^2u)\Phi_\eps^3}{\eps^2}
              [\Heps(1-\mu_\eps)-\mu_\eps s]\biggr\}\\
 &+\frac{4(n-1)\cos u(1-\mu_\eps)\Phi_\eps^3}{\eps^2}
              [\Heps(1-\mu_\eps)-\mu_\eps s]\\
 &-\frac{2\mu_\eps^3\Phi_\eps^6}{\eps^2}
              [\Heps(1-\mu_\eps)-\mu_\eps s]^2.
\end{aligned}
\]
Its quadratic coefficient is
\[
 \frac12\partial_s^2p_\eps
 =-\frac{\cos^2u\Phi_\eps^2}{\eps^2}
       (2-4\cos^2u+2\cos^2u\mu_\eps+\mu_\eps^2).
\]
Since $\Phi_\eps^2=\cos^2u+\eps^2\sin^2u$ and $0<\eps\le1$,
$\cos u\le\mu_\eps\le1$. Hence
\[
\begin{aligned}
 2-4\cos^2u+2\cos^2u\mu_\eps+\mu_\eps^2
 &=1+(1-\cos u)^2(2\cos u+1)\\
 &\quad+(\mu_\eps-\cos u)(\mu_\eps+\cos u+2\cos^2u)\ge1.
\end{aligned}
\]
In particular, $\tfrac12\partial_s^2p_\eps\le-\cos^2u\Phi_\eps^2/\eps^2$.
Rationalization gives
\[
 \frac{(1-\mu_\eps)\Phi_\eps^3}{\eps^2}
 =\frac{\sin^2u\Phi_\eps^2}{\Phi_\eps+\cos u}.
\]
Differentiating the polynomial at $s=0$ therefore gives
\[
\begin{aligned}
 \partial_sp_\eps(0)={}&\Heps\biggl\{
  \cos^2u\sin^2u
  +\frac{\cos u(3-4\cos^2u)\sin^2u\Phi_\eps^2}{\Phi_\eps+\cos u}\\
 &\qquad+\frac{4\cos^4u\sin^2u\Phi_\eps}{\Phi_\eps+\cos u}\biggr\}
 -2\cos^2u(2-\mathsf E_{u,n})
 -\frac{4(n-1)\cos^2u\sin^2u\Phi_\eps}{\Phi_\eps+\cos u}.
\end{aligned}
\]
Since $0<\cos u\le\Phi_\eps\le1$, $|3-4\cos^2u|\le3$ and
$|\Heps|\le2n+1$,
\[
\begin{aligned}
 |\partial_sp_\eps(0)|
 &\le C\bigl\{\cos^2u(1+|\mathsf E_{u,n}|)
                  +\cos u\Phi_\eps+\cos^4u\bigr\}\\
 &\le C\cos u\Phi_\eps(1+|\mathsf E_{u,n}|).
\end{aligned}
\]
For the constant coefficient, discard its nonpositive square to obtain
\[
 p_\eps(0)
 \le C\frac{\eps^2\cos u\Phi_\eps\sin^4u}{(\Phi_\eps+\cos u)^2}
 \le C\eps^2,
\]
using $4\cos u\Phi_\eps\le(\Phi_\eps+\cos u)^2$.
On $\Sigma^*$, $\cos u>0$ makes the constant second
derivative negative. Completing the square gives
\[
\begin{aligned}
 p_\eps(s)
 &=\frac{\partial_s^2p_\eps}{2}
     \left(s+\frac{\partial_sp_\eps(0)}{\partial_s^2p_\eps}\right)^2
       +p_\eps(0)-\frac{|\partial_sp_\eps(0)|^2}{2\partial_s^2p_\eps}\\
 &\le C\eps^2(1+\mathsf E_{u,n}^2).
\end{aligned}
\]
Taking $s=\mathsf E_H$ yields \eqref{eq:joint-global-square}.

Under the dilation flow, change of variables gives
\[
 \int_{\Dil_{e^s}\Omega}\Hpr_j\dd x
 =e^{(2n+2)s}\int_\Omega\Hpr_j(\Dil_{e^s}x)\dd x.
\]
Twice differentiating this identity and the scaling law for $P_\eps$ gives
\[
 \begin{aligned}
 \delta_{\Dil}^2\int_\Omega\Hpr_j\dd x
 &=\int_\Omega\{\Dil^2\Hpr_j+2(2n+2)\Dil\Hpr_j+(2n+2)^2\Hpr_j\}\dd x,\\
 \delta_{\Dil}^2P_\eps
 &=(2n+1)^2P_\eps-2(2n+1)\Lambda_\eps+\eps\partial_\eps\Lambda_\eps.
 \end{aligned}
\]
Here $\Dil^2\Hpr_j=\Dil(\Dil\Hpr_j)$. Stationarity gives
$(2n+1)P_\eps-\Lambda_\eps=\int_\Omega(\Dil\Hpr_j+(2n+2)\Hpr_j)\dd x$.
Substituting these formulas and $I_{\Dil}=(2n+2)V$ into the definitions of
$\mathsf D$ and $\mathcal R_{\Dil}(\Dil)$ yields
\[
 \begin{aligned}
 \mathsf D&=(2n+1)\Delta P_\eps-\Lambda_\eps
       -\int_\Omega\{\Dil\Hpr_j+(2n+2)(\Hpr_j-2n)\}\dd x,\\
 \mathcal R_{\Dil}(\Dil)&=-(2n+1)\Lambda_\eps+\eps\partial_\eps \Lambda_\eps\\
     &\quad-\int_\Omega\{\Dil^2\Hpr_j+(2n+3)\Dil\Hpr_j+(2n+2)(\Hpr_j-2n)\}\dd x.
 \end{aligned}
\]
Since $\kappa=I_F/I_{\Dil}$ and $2n I_{\Dil}=(2n+1) P_0(\Omega)+\mathsf D$,
\[
 \frac{2n\mathsf D\,I_F^2}{I_{\Dil}(2n+1) P_0(\Omega)}
 =\kappa^2\frac{2n I_{\Dil}\mathsf D}{(2n+1) P_0(\Omega)}
 =\kappa^2\left(\mathsf D+\frac{\mathsf D^2}{(2n+1) P_0(\Omega)}\right).
\]
Subtracting the formulas for $\mathcal R_{\Dil}(\Dil)$ and $\mathsf D$
gives \eqref{eq:dilation-residual-combined}.

\section{The stationary Hessian for normal variations}
\label{app:graph-hessian}

Let $h$ be a smooth ambient function and suppose
$P_\eps-\int_\Omega h\dd x$ is stationary on a smooth patch.
All speeds below have compact support in its noncharacteristic part.
Write $\zeta=\langle\None,T\rangle$, $T^\top=T-\zeta\None$ and
$K_1=|\Aone|^2+\Ric^1(\None,\None)$.
For the $g_1$-normal geodesic variation of speed $v$, its surface
Jacobian $\mathcal J_s$ and normal component satisfy
\[
 \dot\zeta=-T^\top v,\qquad
 \ddot\zeta=-\zeta|\nabla^1_\Sigma v|^2
             +2v\langle\nabla^1_\Sigma\zeta,\nabla^1_\Sigma v\rangle,
 \qquad \dot{\mathcal J}=\Hone v,\quad
 \ddot{\mathcal J}=|\nabla^1_\Sigma v|^2+({\Hone}^2-K_1)v^2.
\]
The first two identities follow by differentiating normal
orthogonality and using that $T$ is a unit Killing field.
For a density $f(\zeta)$, first variation gives
$h=f\Hone+\operatorname{div}^1_\Sigma(f'T^\top)$.
Moreover,
$\operatorname{div}^1_\Sigma T^\top=-\zeta\Hone$ and
$\Delta^1_\Sigma\zeta+K_1\zeta=T^\top\Hone$.
Expanding the second variation gives the cross term
$2f'v\langle\nabla^1_\Sigma\zeta,\nabla^1_\Sigma v\rangle
-2f'\Hone vT^\top v$.
Integrating it by parts, using stationarity and the last two
identities, and polarizing gives
\begin{equation}\label{eq:full-normal-hessian}
 \begin{aligned}
 \mathcal B_h(v,w)=\int\bigl\{&
 (f-\zeta f')\langle\nabla^1_\Sigma v,\nabla^1_\Sigma w\rangle
 +f''(T^\top v)(T^\top w)\\
 &-[(f-\zeta f')K_1+f''|\nabla^1_\Sigma\zeta|^2
                         +\partial_{\None}h]vw\bigr\}\dd\sigma_1.
 \end{aligned}
\end{equation}
For $f(\zeta)=\sqrt{1+(\eps^2-1)\zeta^2}=\Phi_\eps$,
$f-\zeta f'=\Phi_\eps^{-1}$ and
$f''=(\eps^2-1)\Phi_\eps^{-3}$.
Taking instead $f(\zeta)=\sqrt{1-\zeta^2}=\cos u$ on this patch,
and subtracting the corresponding forms with their normal forcing
terms removed, recovers \eqref{eq:common}: use
$T^\top=(\cos u)e_0$ and
$|\nabla^1_\Sigma u|^2=|\grad u|^2+|\Amer|^2$.

The same normal formula also gives a direct derivation of
\eqref{eq:normal-horizontal-comparison}. Put $c=\cos u$, $s=\sin u$
and $a=\Amer=e_0u$. The block decomposition \eqref{eq:blocks},
the Ricci formula and \eqref{eq:bounded-potential} give
\begin{equation}\label{eq:transport-intermediate}
 c^2q_H=K_1-a^2-2|\grad u|^2,\qquad
 c\Delta_Hu+s(K_1-a^2)=c^2e_0\Hhor,
 \quad \Delta_H=\divh\grad.
\end{equation}
For the second identity, use
$\operatorname{div}^1_\Sigma Z=\divh Z+\tan u\langle\grad u,Z\rangle$
for horizontal tangent $Z$, together with
$\operatorname{div}^1_\Sigma e_0=-s\Hhor$ and
$\Hone=c\Hhor+a$, in the Killing identity
$\Delta^1_\Sigma s+K_1s=c\,e_0\Hone$.
For $f(\zeta)=\sqrt{1-\zeta^2}$, removing the normal forcing term
from \eqref{eq:full-normal-hessian} gives
\[
 \Qonehor(v,w)=\int c^{-2}
 \left\{\langle\grad v,\grad w\rangle
       -(K_1-|\nabla^1_\Sigma u|^2)vw\right\}\dd\sigma_H.
\]
Expand $\Qhor(v/c,w/c)$ and integrate its mixed gradient term by
parts. The two identities in \eqref{eq:transport-intermediate}
and $|\nabla^1_\Sigma u|^2=|\grad u|^2+a^2$ give
\[
 \Qonehor(v,w)-\Qhor(v/c,w/c)
   =\int\frac{s}{c^2}(e_0\Hhor)vw\dd\sigma_H.
\]
All integrations here have compact noncharacteristic support;
the global cutoff passage is the one in
Lemma~\ref{lem:global-identities}.

\Needspace{22\baselineskip}

\end{document}